\documentclass[11pt]{article}
\usepackage{appendix}
\usepackage{mathtools}
\usepackage[T1]{fontenc}
\usepackage{amsfonts}
\usepackage{amsmath}
\usepackage{amssymb}
\usepackage{amsthm}
\usepackage{bbm}
\usepackage{bm}
\usepackage{mathrsfs}
\usepackage{color}
\usepackage{pdfsync}
\usepackage{enumitem}
\usepackage{xcolor} %
\definecolor{darkgreen}{rgb}{0,0.5,0} %
\definecolor{darkblue}{rgb}{1,0,0} 

\usepackage[colorlinks=true, linkcolor=blue, citecolor=blue, urlcolor=blue]{hyperref}
\makeatletter
\renewcommand\@makefnmark{%
  \hbox{\textsuperscript{\textcolor{blue}{\@thefnmark}}}}
\makeatother
\usepackage{tikz}

\DeclareMathOperator*{\argmax}{argmax}

\DeclareMathOperator*{\spt}{supp}

\newcommand{\RR}{\mathbb{R}}
\newcommand{\R}{\RR}
\newcommand{\Z}{\mathbb{Z}}
\newcommand{\EE}{\mathbb{E}}
\newcommand{\PP}{\mathbb{P}}
\newcommand{\NN}{\mathbb{N}}

\newcommand{\eps}{ \varepsilon}

\newcommand{\proofreadhere}{}

\let\leqslant\leq
\let\geqslant\geq
\usepackage[margin=31 mm]{geometry}
\newcommand{\mykill}[1]{}
\usepackage[capitalize, noabbrev, nosort]{cleveref}
\crefname{equation}{}{} %

\theoremstyle{plain}
\newtheorem{theorem}{Theorem}[section]
\newtheorem{proposition}[theorem]{Proposition}
\newtheorem{lemma}[theorem]{Lemma}
\newtheorem{corollary}[theorem]{Corollary}
\theoremstyle{definition}
\newtheorem{definition}[theorem]{Definition}
\newtheorem{remark}[theorem]{Remark}
\newtheorem{example}[theorem]{Example}
\newtheorem{assumption}[theorem]{Assumption}
\newtheorem{standingassumption}[theorem]{Standing assumption}
\crefname{assumption}{Assumption}{Assumptions}
\Crefname{assumption}{Assumption}{Assumptions}
\crefname{standingassumption}{Assumption}{Assumptions}
\Crefname{standingassumption}{Assumption}{Assumptions}

\theoremstyle{remark}
\AddToHook{env/theorem/begin}{\crefalias{section}{theorem}}
\AddToHook{env/proposition/begin}{\crefalias{theorem}{proposition}}
\AddToHook{env/lemma/begin}{\crefalias{theorem}{lemma}}
\AddToHook{env/corollary/begin}{\crefalias{theorem}{corollary}}
\AddToHook{env/definition/begin}{\crefalias{theorem}{definition}}
\AddToHook{env/remark/begin}{\crefalias{theorem}{remark}}
\AddToHook{env/example/begin}{\crefalias{theorem}{example}}
\AddToHook{env/assumption/begin}{\crefalias{theorem}{assumption}}
\AddToHook{env/standingassumption/begin}{\crefalias{theorem}{standingassumption}}

\crefname{theorem}{Theorem}{Theorems}
\crefname{proposition}{Proposition}{Propositions}
\crefname{lemma}{Lemma}{Lemmas}
\crefname{corollary}{Corollary}{Corollaries}
\crefname{definition}{Definition}{Definitions}
\crefname{remark}{Remark}{Remarks}
\crefname{example}{Example}{Examples}
\crefname{assumption}{Assumption}{Assumptions}
\crefname{standingassumption}{Assumption}{Assumptions}

\newlist{myenum}{enumerate}{3}
\setlist[myenum,1]{label={\rm (H\arabic*)},
                   ref  ={\rm (H\arabic*)}}
\crefname{myenumi}{property}{properties}

\renewenvironment{thebibliography}[1]{%
\begin{oldthebibliography}{#1}%
\setlength{\baselineskip}{.9em}
\linespread{1}
\small
\setlength{\parskip}{.40ex}%
\setlength{\itemsep}{.30em}%
}%
{%
\end{oldthebibliography}%
}

\definecolor{grape}{rgb}{0.43, 0.17, 0.71}

\begin{document}

\title{Finite-sample bounds for regularized optimal transport}
\date{\today}
\author{  
  Alberto Gonz{\'a}lez-Sanz%
  \thanks{Department of Statistics, Columbia University, \textcolor{blue}{ag4855@columbia.edu}}   \and  Marcel Nutz%
  \thanks{Departments of Mathematics and Statistics, Columbia University, \textcolor{blue}{mnutz@columbia.edu}. Research supported by NSF Grants DMS-2106056, DMS-2407074.} \and Austin J.~Stromme%
  \thanks{Department of Statistics, CREST, ENSAE, IP Paris, \textcolor{blue}{austin.stromme@ensae.fr}.}
  }
  
\maketitle 

\vspace{-1.5em}
\begin{abstract}
We study the sample complexity of regularized optimal transport for general convex regularizations including the Kullback--Leibler divergence and $L^p$ penalties. Our main results are non-asymptotic bias and variance bounds for the empirical cost, with explicit dependence on the regularization parameter and on the intrinsic dimension of the marginals. Our approach simultaneously improves, unifies, and extends existing finite-sample bounds. In particular, we improve the state of the art for entropic optimal transport, and we obtain the first fully quantitative results for $L^p$ regularization with $1<p<\infty$. For the quadratic transport cost, we deduce that quadratically regularized optimal transport (i.e., $L^2$ regularization) estimates the unregularized optimal transport cost at rate $n^{-2/(d+4)}$, the fastest non-asymptotic rate currently available for any estimator based on regularized optimal transport.
\end{abstract}

\vspace{.3em}

{\small

}
 \vspace{.1em}

\setcounter{tocdepth}{2}
\tableofcontents

\section{Introduction}

Optimal transport (OT) is a widely used method for comparing probability distributions (e.g., \cite{villani2008optimal,santambrogio2015optimal,peyre2019computational}). Given probability measures $P$ and $Q$ on $\R^d$ and a cost function $c:\R^d\times \R^d \to \R$, the optimal transport problem is
\[
{\rm OT}:={\rm OT}(P, Q)= \inf_{\pi \in \Pi(P, Q)} \int c\, d\pi,
\]
where $\Pi(P, Q)$ denotes the set of all couplings between $P$ and $Q$. 
In many applications of OT, the marginals $P$ and $Q$ are not explicitly known; instead, one has access to i.i.d.\ samples. Given independent random variables $X_1, \dots, X_n, Y_1, \dots, Y_n$ with $X_i\sim P$ and $Y_i\sim Q$, denote the associated empirical measures by $\widehat{P} = \frac{1}{n}\sum_{i=1}^n \delta_{X_i}$ and $\widehat{Q} = \frac{1}{n}\sum_{j=1}^n \delta_{Y_j}$. Then the empirical OT problem is 
$$
\widehat{\rm OT}:= {\rm OT}(\widehat{P},\widehat{Q}) = \inf_{\pi\in \Pi( \widehat{P} ,   \widehat{Q} )} \int c \,d\pi.
$$
It is well known that OT suffers from the \emph{curse of dimensionality}: in general, the rate at which the empirical cost $\widehat{\rm OT}$ converges to its population counterpart ${\rm OT}$ is of order $n^{-1/d}$~\cite{Dudley.69,Fournier.2014.PTRF}, limiting the applicability of OT in high-dimensional problems.

Regularization by a strictly convex penalty is a common approach to remedy the curse of dimensionality while also reducing the computational burden of OT. Given a regularization parameter $\eps>0$ and a strictly convex function $\phi:\R_+\to\R$, we study the regularized optimal transport (ROT) problem
\begin{align}\label{rotIntro}
  {\rm OT}_{\phi,\eps} :=  {\rm OT}_{\phi,\eps}(P,Q)=  \inf_{\pi\in \Pi( P ,   Q )} \int c \,d\pi  
  +{\eps } D_\phi(\pi|P\otimes Q), 
\end{align} 
where 
$$ 
  D_\phi(\pi|P\otimes Q):= \int \phi\biggl( \frac{d \pi}{ d( P \otimes   Q )}  \biggr) d( P \otimes   Q ),
$$
as well as its empirical counterpart $\widehat{\rm OT}_{\phi,\eps} :=  {\rm OT}_{\phi,\eps}(\widehat P,\widehat Q)$. The most popular choice for~$\phi$ is $\phi_1(t):=t\log t$, leading to Kullback--Leibler divergence $D_{\text{KL}}$ and the \emph{entropic optimal transport} (EOT) problem ${\rm EOT} := {\rm OT}_{\phi_1,\eps}$. Indeed, EOT was popularized on computational grounds---the solution to its dual problem can be computed efficiently via Sinkhorn's algorithm~\cite{Cuturi.2013.Neurips}, which enjoys linear convergence (see \cite{FRANKLIN.1989,Carlier.2022.SIOPT,ChizatDelalandeVaskevicius.25}, among many others). Starting with~\cite{genevay.2019.PMLR}, several works demonstrated that entropic regularization
also evades the
curse of dimensionality (for fixed $\eps > 0$). In general, the dual problem of~\eqref{rotIntro} is
 \begin{align}
    \label{dualIntro}
     \sup_{(f,g)\in L^\infty(  P )\times L^\infty(  Q )} \Phi_{\phi,\eps}(f,g ),
\end{align}
$$ \Phi_{\phi,\eps}(f,g ) :=\int\biggl\{ f(x) + g(y) - \eps\cdot\psi\biggl(\frac{f(x)+g(y)-c(x,y)}{\eps} \biggr) \biggr\}d(  P \otimes   Q )(x,y),$$
where $\psi(s):=\sup_{t\geq 0} \{ st-\phi(t)\}$ is the convex
conjugate of $\phi$. For EOT, the conjugate is $\psi(s)=e^{s-1}$, and the smoothness of this dual problem is the basis of~\cite{genevay.2019.PMLR}. While the constant in their bound was exponential in $1/\eps$, this was subsequently improved to polynomial for smooth costs by \cite{MenaWeed.2019.Nips} and to Lipschitz costs by \cite{Stromme.24}, whereas \cite{rigollet2022samplecomplexityentropicoptimal} obtained a bound with exponential dependence for bounded measurable cost.

Starting with \cite{Muzellec.2017.AAAI,blondel18quadratic,EssidSolomon.18,Lorenz.2019}, a growing literature studies alternative regularizations, in particular quadratically regularized optimal transport (QOT) which uses the squared $L^2$ norm $\phi_2(t) :=\frac{1}{2} t^2$ (or, equivalently up to a constant, the $\chi^2$ divergence $\frac{1}{2} (t^2-1)$). This literature observed that the resulting transport plans exhibit sparsity, in the sense that the support converges to the support of the unregularized OT plan (which is typically a graph) as $\eps\to0$, in contrast to the EOT plans which have full support for any $\eps>0$.  Sparsity is shared by the $L^p$ regularizer $\phi_p(t) = \frac{1}{p} t^p$ for $p\in(1,\infty)$ (or equivalently the $p$-Tsallis divergence), which for $p\in(1,2)$ interpolates between EOT and QOT. See also \cite{WieselXu.24,GonzalezSanzNutz2024.Scalar,Nutz.2024,gonzalezsanz2026sharplocalsparsityregularized} for recent analytic results on sparsity, and \cite{GonzalezSanzNutzRiveros.25,GonzalezSanzNutzRiveros.26} for algorithms with linear convergence guarantees.

For such regularizers, the conjugate 
is $\psi_p(s) = \frac{(s)_+^q}{q}$ if $p>1$, where $q=\frac{p}{p-1}$ denotes the conjugate exponent and $(s)_+=\max(0,s)$ the positive part. The
limited smoothness of the resulting dual problem~\eqref{dualIntro} might lead one to conjecture that such ROT problems have significantly worse sample complexity than EOT, or even suffer from a direct curse of dimensionality, as suggested by the bounds in \cite{BayraktarEckstein.2025.BJ}. Recently, \cite{gonzalezsanz.2025.sparseregularizedoptimaltransport, GonzalezSanzDelBarrioNutz.25} obtained central limit theorems with the usual rate $\sqrt{n}$ for $p\in(1,2]$, giving the first indications that less-smooth penalties still have a strong regularizing effect in terms of sample complexity. Specifically, \cite{gonzalezsanz.2025.sparseregularizedoptimaltransport} covered Lipschitz costs $c$ and divergences for which $\psi\in\mathcal{C}^2$---thus including $p\in[1,2)$, but not $p\geq2$---and exploited this regularity in a Z-estimation approach. QOT, corresponding to the boundary case of $p=2$ where $\psi\in\mathcal{C}^{1,1}$, was covered in \cite{GonzalezSanzDelBarrioNutz.25}, which restricts to the quadratic cost $c(x,y)=\|x-y\|^2$ for technical reasons. These central limit theorems imply sample complexity results at rate $n^{-1/2}$, which are however qualitative in that the leading constants are not quantified. The latter is inherent to the proof technique using compactness arguments; in
particular, a direct comparison with the
sample complexity of EOT was not possible. For $p>2$, parametric sample complexity was not known before the present work, regardless of quantification of constants.

For further background on statistical
optimal transport, we refer the reader to the comprehensive surveys~\cite{delbarrio.et.al.2025.survey,balakrishnan.2025.statisticalinferenceoptimaltransport,Chewi_Niles-Weed_Rigollet}.

\subsection{Summary of contributions}
We develop new approaches to sample complexity which simultaneously improve, unify, and extend the previous results. In particular, we

\begin{itemize}
    \item recover and improve the sharpest known sample complexity results for EOT;
    \item give the first fully quantitative results for $L^p$ regularization with $p\in(1,2]$, and in particular for QOT;
    \item give the first parametric results for general regularizations with conjugate $\psi\in\mathcal{C}^{1}$, including  $L^p$ for arbitrary $p\in(1,\infty)$.
\end{itemize}

The following paragraphs summarize a selection of the take-aways from our main results in the body of the text; the underlying proof approaches are explained in the subsequent \cref{se:proof-strategies}. For simplicity of exposition, we focus on the $p$-Tsallis (or $L^p$)  entropies for $p\in[1,\infty)$, including EOT for $p=1$ and QOT for $p=2$, even though some results in the body are more general. Moreover, we may focus on the bias: the mean squared error admits the bias--variance decomposition~\eqref{eq:bias-variance-decomp} and we will see in \cref{pr:variance} that the variance admits a very general $n^{-1}$ rate with coefficient independent of~$\eps$ and~$d$.

Our first main result establishes an $n^{-\frac{1}{2}}$ rate of the bias (the so-called slow rate) with polynomial dependence on the regularization~$1/\eps$. For $p=1$ (EOT), this result recovers the best known constant (see~\cite{Stromme.24}). For $p\in(1,2]$, and in particular for QOT, it is the first parametric sample complexity result with quantification of the constant, and for $p>2$, it is the first parametric rate at all. 

\begin{theorem}[Slow rate]\label{Theorem:Slow-bias}
    Assume that  $P$ and $Q$ are supported on bounded sets $\Omega$ and $\Omega'$ with Minkowski dimensions ${\rm dim}_M(\Omega)$ and ${\rm dim}_M(\Omega')$ (as defined in \cref{sect:Lp-Entropies-cost-adaptation}). Assume further that $c$ is Lipschitz. Then, for every $ \alpha> \min( {\rm dim}_M(\Omega),{\rm dim}_M(\Omega'))  $,  
    $$\bigl|\EE\bigl[\widehat{\rm OT}_{\phi_p,\eps}\bigr] -{\rm OT}_{\phi_p,\eps}\bigr| \lesssim \frac{\eps^{-\frac{\alpha }{  2(1+(p-1) \alpha)}}}{\sqrt{n}}, \quad \text{for }p\in[1,\infty).$$
\end{theorem}

\cref{Theorem:Slow-bias} also shows for the first time that the constant for ROT depends only on the \emph{minimum} of the dimensions of $P$ and $Q$; for instance, in the extreme case of semi-discrete transport where one marginal is discrete, the dimension disappears completely (\cref{section:Complexity-semidiscrete}). 
 This phenomenon was first identified 
 for unregularized OT,
 where it was called the lower complexity adaptation principle (LCAP)~\cite{Hundrieser.et.al.AHIP.2024}
 and then
for the entropically regularized
 problem, where it was dubbed minimum intrinsic dimension scaling~\cite{Stromme.24,Groppe-Shayan-LCA-JMLR}.
 \cref{Theorem:Slow-bias} shows that
 this phenomenon holds for general ROT problems.

A second, more subtle, implication
of~\cref{Theorem:Slow-bias} is that the statistical complexity of EOT ($p = 1$)
is significantly more sensitive to the dimension than estimators with $L^p$ penalties for $p>1$. Indeed, the exponent of $1/\eps$ is $\alpha/2$ for $p=1$, hence linear in the dimension $\alpha$, whereas for fixed $p>1$, the exponent tends to the constant $\frac{1}{2(p-1)}$ for large~$\alpha$. This might appear counterintuitive at first glance, given that EOT leads to the smoothest plans. The resolution of this paradox lies in the trade-off between statistical complexity and regularization bias (see \cref{section:tradeoff}): the meaning of $\eps$ is not directly comparable across the problems. 

For the first time, our quantitative results enable a more meaningful way of comparing different regularizations by directly evaluating the estimation of the unregularized OT problem. Indeed, tuning $\eps$ as a function of $n$, \cref{Theorem:Slow-bias} implies a rate for the approximation of the unregularized OT problem. For instance, for EOT, we recover 
$$\bigl|\EE\bigl[\widehat{\rm EOT}\bigr] -{\rm OT}\bigr| \lesssim n^{-\frac{1}{{\alpha}+2}}\,
(\log n)^{\frac{{\alpha}}{{\alpha}+2}},$$
and for QOT we obtain the comparable rate
$$\bigl|\EE\bigl[\widehat{\rm QOT}\bigr] -{\rm OT}\bigr| \lesssim n^{-\frac{1}{{\alpha}+2}}.
$$
However, it turns out that, specifically for QOT, this rate can be significantly improved in the key case where $c$ is the quadratic cost $c(x,y)=\|x-y\|^2$. Here, our result (see \cref{prop:improved_qot_quadratic_cost}) is
$$
\bigl|\EE\bigl[\widehat{\rm QOT}\bigr] -{\rm OT}\bigr|
\lesssim n^{-\frac{2}{d + 4}}.
$$
Remarkably, this bound is strictly better than any previously known rate for regularized OT derived from a non-asymptotic bound for general~$\eps$---EOT included---and thus positions QOT as a competitive alternative. Moreover, this bound almost achieves $n^{-2/d}$, which is the minimax rate for estimating ${\rm OT}$ in this setting (for $d>4$).

To conclude our discussion of the slow rate, we mention a central limit theorem for the empirical cost (\cref{Theorem:CLT}), centered at its population counterpart and with parametric rate $n^{-\frac{1}{2}}$. It generalizes the central limit theorems of~\cite{gonzalezsanz.2025.sparseregularizedoptimaltransport, GonzalezSanzDelBarrioNutz.25} for the empirical cost to more general divergences and transport costs, with a completely different proof. In EOT, central limit theorems are well established: \cite{MenaWeed.2019.Nips} and \cite{delBarrioEtAl.2023.SIMODS,Goldfeld.et.al.2022.statisticalinferenceregularizedoptimal} proved central limit theorems for the fluctuations and empirical cost, respectively, assuming smooth transport costs, later generalized to bounded transport costs in \cite{GonzalezSanz.2023.Beyond}. See also \cite{GonzalezSanz.Loubes.Weed.2024.weaklimits,Goldfeld.et.al.2024.EJS}   for central limit theorems for the empirical plans and potentials, \cite{li2025samplecomplexityweaklimits} for multimarginal EOT, and \cite{Bigot.Cazelles.Papadakis.CLTEOT.2019} for the discrete case.

\bigskip

Our second main result establishes an $n^{-1}$ rate for the bias, the so-called fast rate, under stronger assumptions on the marginals and for $p\in [1,2]$. This faster rate comes at the expense of a larger constant (compared to~\cref{Theorem:Slow-bias}); the constant remains polynomial in $1/\eps$ but is now exponentially growing in the dimension even for $p>1$. Our result improves upon the best previous result for EOT; for other regularizations, there is no comparable previous result.

\begin{theorem}[Fast rate]\label{Theorem:fast-bias}
    Assume that  $P$ and $Q$ are supported on bounded convex sets with densities bounded away from zero and infinity. Assume further that $c$ is Lipschitz. Then
    $$\bigl|\EE\bigl[\widehat{\rm OT}_{\phi_p,\eps}\bigr] -{\rm OT}_{\phi_p,\eps}\bigr| \lesssim \frac{ \eps^{-(3 d+2+\frac{d}{1+d(p-1)})}}{n } , \quad \text{for $p\in [1,2]$}.
    $$ 
\end{theorem}
For $p=1$, this specializes to 
$$
\bigl|\EE\bigl[\widehat{\rm EOT}\bigr] -{\rm EOT}\bigr| \lesssim \frac{   \eps^{-(4d+2)}  }{ n}
$$
and this leading term is strictly smaller than the $\eps^{-9d-6}$ obtained in \cite[Theorem~12]{Stromme.24}. Moreover, our result allows for the practical case of marginal supports with boundary, whereas \cite{Stromme.24} assumes manifolds without boundary. We note that, although
\cref{Theorem:fast-bias} involves the ambient dimension~$d$, it is consistent with
the lower complexity adaptation principle because the
measures are indeed supported on $d$-dimensional sets.

As we discuss in more detail below,
the improved rate in \cref{Theorem:fast-bias}
is enabled by a fine analysis of the convergence
of the dual potentials. It shows that the empirical potentials $(\widehat{f}, \widehat{g})$ converge to the potentials $(f_\eps,g_\eps)$  at a parametric rate.
\begin{theorem}[Rate for potentials]\label{Theorem:potentials-intro}
    Assume that  $P$ and $Q$ are supported on bounded convex sets with densities bounded away from zero and infinity. Assume further that $c$ is Lipschitz. Then 
    $$\EE\biggl[ \bigl\|\bigl(\widehat{f} - f_\eps \bigr) \oplus \bigl( \widehat{g} - g_\eps \bigr) \bigr\|_{L^2(\widehat{P}\otimes \widehat{Q})}^2  \biggr] \lesssim \frac{\eps^{ -( 6 d+4 +\frac{d}{1+d(p-1)}) }}{n}   ,\quad \text{for $p\in [1,2]$}.$$
\end{theorem}

We mention that the convexity assumption on the marginal supports could be relaxed to connected domains with Lipschitz boundary in both \cref{Theorem:fast-bias,Theorem:potentials-intro}, at the expense of an additional domain-dependent constant. Regarding prior results, the central limit theorems of \cite{gonzalezsanz.2025.sparseregularizedoptimaltransport, GonzalezSanzDelBarrioNutz.25} also imply parametric convergence rates for the potentials in their respective settings, but as mentioned above, their techniques do not allow for any quantification of the constant, while that is our main objective
in the current work.

\subsection{Proof strategies}\label{se:proof-strategies}

The following paragraphs highlight some of the techniques used in our proofs. In view of the bias--variance decomposition~\eqref{eq:bias-variance-decomp} we start with a general bound (\cref{pr:general}) controlling the bias in terms of the potentials and the density $\rho_\eps$ (cf.~\eqref{eq:dual-primal}) of the optimal regularized plan~$\pi_\eps$,
\begin{equation}
    \label{bias-intro}
    (\EE[\widehat{\rm OT}_{\phi,\eps}]- {\rm OT}_{\phi,\eps})^2 \leq 2\min_{a,b\in \R}\frac{\EE\Bigl[\|f_\eps-\widehat f-a\|_{L^2(\widehat{P})}^2+ \|g_\eps-\widehat g-b\|_{L^2(\widehat{Q})}^2 \Bigr] {\rm Var}_{{P\otimes Q}} [\rho_\eps]}{n}.
\end{equation}
The proof is based on the convexity technique introduced in~\cite{Stromme.24}. However, in contrast to the bound in~\cite[Eq.~(5.4)]{Stromme.24}, \eqref{bias-intro} involves a term including the potentials. This term could immediately be replaced by a constant upper bound, but the term itself tends to zero, and this will be fundamental for some of the developments below, including the central limit theorem in \cref{Theorem:CLT} where \eqref{bias-intro} is used to show that the limiting Gaussian variable is centered.

On the other hand, we bound the variance by
$$ {\rm Var}\bigl[\widehat{\rm OT}_{\phi,\eps}\bigr]\leq \frac{4\|c\|_\infty^2}{n}
$$
in \cref{pr:variance}. The proof uses a strategy based on the Efron--Stein inequality that was introduced by \cite{delBarrio.Loubes.2019.AoP} in the context of unregularized optimal transport. Note that, in contrast to the variance bound established in \cite{Stromme.24}, our bound does not depend on the regularization parameter~$\eps$. The same Efron--Stein strategy is also used in the proof of the central limit theorem.

\bigskip

Next, we turn to the fast rate (\Cref{theorem:fast-bias-general}) and the convergence of the potentials (\Cref{theorem:potentials}), where the proofs are substantially more involved and arguably the most innovative. We begin by defining the convex function 
$$\Gamma(t) = -\widehat{\Phi}_{\phi,\eps}\bigl( (1-t)({\widehat{f}},{\widehat{g}} )  + t\, ({f}_\eps,g_\eps)\bigr), \quad t\in [0,1], $$
where $\widehat{\Phi}_{\phi,\eps}$ is the empirical dual objective function. As shown in \Cref{lemma:taylor}, this function admits the second-order Taylor expansion
$$  \Gamma(t) =-\widehat{\rm OT}_{\phi,\eps} +  \int_{0}^t \int_{0}^s   \frac{1}{\eps} \int  \Bigl(\bigl( \widehat{f} - f_\eps \bigr) \oplus \bigl( \widehat{g} - g_\eps \bigr) \Bigr)^2\psi''(\gamma_r)  d(\widehat{P}\otimes \widehat{Q}) dr ds , $$
where 
$$ \gamma_r =\frac{\bigl((1-r){\widehat{f}}+ r f_\eps\bigr)\oplus \bigl((1-r){\widehat{g}}+ r g_\eps\bigr) - c}{\eps} .$$
The core of our argument is to establish a coercivity bound of the form
\begin{equation}
    \label{eq:coercivity-intro}
    \widehat{\beta}\int  \Bigl(\bigl( \widehat{f} - f_\eps \bigr) \oplus \bigl( \widehat{g} - g_\eps \bigr) \Bigr)^2  d(\widehat{P}\otimes \widehat{Q})\lesssim  \int  \Bigl(\bigl( \widehat{f} - f_\eps \bigr) \oplus \bigl( \widehat{g} - g_\eps \bigr) \Bigr)^2\psi''(\gamma_r)  d(\widehat{P}\otimes \widehat{Q}) ,
\end{equation}
for some $\widehat{\beta} > 0$, possibly
random, and for sufficiently small $r$. This is related to an empirical version of a PL inequality for the dual problem. (A deterministic PL inequality was separately presented in~\cite{GonzalezSanzNutzRiveros.26}, but we mention that key ideas first arose in the context of the present work.) An inequality related to \eqref{eq:coercivity-intro} was established in~\cite{Stromme.24}, but for the semi-dual objective which involves only one of the two potentials. That approach fails for the more general regularizations
considered here, and also leads to slower rates for EOT. In the present work, we develop a finer analysis studying the full dual objective $\Gamma(t)$.

The first step towards~\eqref{eq:coercivity-intro} is to prove that
\begin{equation}\label{eq:bound-coer-one-side-intro}
    \widehat{\alpha} \, {\rm Var}_{\widehat{P}}\bigl[\widehat f - f_\eps\bigr] {\bf 1}_{\mathcal{E}_n} \lesssim  \int  \Bigl(\bigl( \widehat{f} - f_\eps \bigr) \oplus \bigl( \widehat{g} - g_\eps \bigr) \Bigr)^2\psi''(\gamma_r)  d(\widehat{P}\otimes \widehat{Q}) 
\end{equation}
for a controlled random variable $\widehat{\alpha} > 0$, where the event $\mathcal{E}_n$ is defined by
$$\mathcal{E}_n= \biggl\{  \text{there exists $\hat{T}$ s.t.~$\hat{T}_\# P =\widehat{P}$ and $  \|\hat{T}-{\rm I}\|_{L^\infty(P)} \leq C\eps$}\biggr\}.$$ 
To establish \eqref{eq:bound-coer-one-side-intro}, we first demonstrate in \Cref{Lemma:From two-to-one-variance} that for every $u,v$, 
$$   \frac{\widehat{\alpha}_{1}}{n^2}\sum_{\|X_i-X_\ell\| \leq C\eps}     \bigl(u(X_i)- u(X_\ell) \bigr)^2\lesssim \int  \bigl(u  \oplus v \bigr)^2\psi''(\gamma_r)  d(\widehat{P}\otimes \widehat{Q}) ,$$
for another controlled random $\widehat{\alpha}_1 > 0$. The left-hand side of this inequality is related to the graph Laplacian of a random geometric graph with edge weights ${\bf 1}_{\|X_i-X_\ell\| \leq \eps}$. To bound it from below, we adapt an approach of \cite{Trillos-manifold-laplace} which was developed for random geometric graphs with i.i.d.\ data on manifolds without boundary. Specifically, under the event $\mathcal{E}_n$, we define the isometry \begin{align*}
    \Pi_T: L^2(\widehat{P})& \to L^2(P) \\
   f &\mapsto  \sum_{i}  f(X_i) {\bf 1}_{T^{-1}(X_i)}=  f\circ T, 
\end{align*}
for $T$ pushing $P$ forward to $\widehat{P}$. In \Cref{lemma:Theta-sigma} we then verify the approximation 
\begin{equation}
    \label{eq:equivalence-between-empirical-isometric}
    \frac{1}{n^2}\sum_{\|X_i-X_\ell\| \leq \delta}     \bigl(u(X_i)- u(X_\ell) \bigr)^2 \approx \int \int_{\mathbb{B}_{{\delta}}(x)} \bigl( (\Pi_T u)(x)-(\Pi_T u)(z) \bigr)^2 dP(x) dP(z) .
\end{equation}
Lower bounds on the right-hand side are closely related to Poincaré inequalities and connectedness of the support.
In particular, we use
a technical result from~\cite{GonzalezSanzNutzRiveros.26} which states
that when $P$ has bounded convex support, then for any $h\in L^2(P)$ and $\delta > 0$,
\begin{equation}
    \label{eq:bound-energy-intro}
    \int \int \bigl(h(x)-h(z)\bigr)^2  d P(x) dP(z)  \lesssim  \delta^{-d-2} \int \int_{\mathbb{B}_\delta(x) } \bigl(h(x)-h(z)\bigr)^2  d P(x) dP(z).
\end{equation}
The estimates \eqref{eq:equivalence-between-empirical-isometric} and \eqref{eq:bound-energy-intro}, together with the fact that $\Pi_T$ is an isometry, imply~\eqref{eq:bound-coer-one-side-intro}. 

The second step is to derive \eqref{eq:coercivity-intro} from \eqref{eq:bound-coer-one-side-intro} and a spectral analysis of the operator associated with the bilinear form  
\[
\bigl((u, v), (\tilde{u}, \tilde{v})\bigr) \mapsto \int \psi''(\gamma_r) (u \oplus v) (\tilde{u} \oplus \tilde{v})  d(\widehat{P} \otimes \widehat{Q});
\]
see \cref{proposition:coercivity}. Finally, we show that $\widehat{\beta}{\bf 1}_{\mathcal{E}_n}$ is close to its population counterpart (see $\beta_{n,\eps}$ in \Cref{theorem:potentials}) with sufficiently high probability. This uses standard arguments of empirical process theory and a control on the $\infty$-Wasserstein distance proved in~\cite{Trillos-Wass-infty}. 

\paragraph{Organization of the paper}  
The remainder of the paper is organized as follows. \Cref{section:Notation-background} introduces both the notation and the main background results needed for our analysis. 
\Cref{Section:Cost} focuses on the statistical complexity of the cost. \cref{Section:Cost-general} presents the most general results, under minimal assumptions on the cost and on the divergences. In \cref{section:Lipschitz-cost} we refine these bounds for Lipschitz costs. \cref{section:Complexity-semidiscrete} is devoted to the semi-discrete setting, where we show that the curse of dimensionality can be avoided.  %
\Cref{sect:Lp-Entropies-cost-adaptation} summarizes the preceding results in the context of Tsallis entropies, while \Cref{section:tradeoff} investigates the trade-off between regularization bias and statistical complexity. The proofs for \cref{Section:Cost} are gathered in \cref{sect:proofs-cost}.

\Cref{sect-complexity-pot} examines the statistical complexity of the potentials.  The proofs for this  section are deferred to \cref{sect-proof-pot}. Finally, \Cref{sect:fast-bias} builds on these estimates to provide bounds on the cost with fast convergence rates.

\section{Preliminaries}\label{section:Notation-background}

The following is in force throughout the paper.

\begin{standingassumption}[Regularization]\label{assumption:divergence} The function $ \phi:[0, \infty)\to \R$ is strictly convex with $\phi(x)/x \to \infty$ as $x\to \infty$. 
Furthermore,  the conjugate $$y\mapsto \psi(y):=\phi^*(y):= \sup_{x\geq 0} \{ xy-\phi(x)\} $$ belongs to $\mathcal{C}^1(\R)$ and satisfies
\begin{enumerate}
    \item $\psi'(x) \to \infty $ as $x\to \infty$;
    \item there exist $t_\psi>0$ and $ \delta_\psi>0$ such that $\psi'(t_\psi)=1$ and $\psi$ is strictly convex on $ [t_\psi-\delta_\psi, \infty)$.
\end{enumerate}  
\end{standingassumption} 

\begin{remark}\label{rem:phi-differentiability}
As can be observed directly
from its definition, $\psi$ is
non-decreasing, so that
$\psi' \geqslant 0$ everywhere. Moreover, under Assumption~\ref{assumption:divergence}, the function $\phi$ is differentiable on $(\psi'(t_\psi-\delta_\psi),\infty)$, and in particular on $[1,\infty)$, due to the strict convexity of $\psi$ on $[t_\psi-\delta_\psi,\infty)$. We note that $\phi'(1)=t_\psi>0$.
\end{remark}

For two probability measures $\mu\ll\nu$ on the same space, we then define
$$ 
  D_\phi(\mu|\nu):= \int \phi\biggl( \frac{d \mu}{ d\nu}  \biggr) d\nu,
$$
and we set $D_\phi(\mu|\nu):=\infty$ if $\mu\not\ll\nu$. We note that $D_\phi$ is a bona fide divergence only if $\phi(1)=0$, which leads to $D_\phi(\nu|\nu)=0$. Changing $\phi$ by an additive constant changes $D_\phi$ and $\psi$ only by an additive constant, and that does not materially affect the results in this paper. We did not include a normalization for $\phi$ in \cref{assumption:divergence} in order to cover both the $L^p$ and $p$-Tsallis regularizations.

Note that our assumptions on the divergence relax those of~\cite{gonzalezsanz.2025.sparseregularizedoptimaltransport}. Indeed, we assume $\psi \in \mathcal{C}^1(\mathbb{R})$ rather than $\psi \in \mathcal{C}^2(\mathbb{R})$, and this is necessary to cover the Tsallis entropies for $p \geq 2$. We also impose the growth condition that $\psi'(x) \to \infty $ as $x\to \infty$, instead of the stronger condition that $\psi'(x) \geq x$ for all sufficiently large $x$.

The following basic result will be used both for the population marginals $(P,Q)$ and the empirical marginals $(\widehat P,\widehat Q)$. For a function $h: \Omega\to\R$, we denote $\|h\|_{\infty, \Omega} = \sup_{x \in \Omega} |h(x)|$, and we simply write $\|h\|_\infty$ when the domain is clear from context.
 
\begin{proposition}\label{pr:ROTprelims}
    Let $P,Q$ be probability measures on $\mathbb{R}^d$ with supports contained in given compact sets $\Omega,\Omega'$. Moreover, let $c:\Omega\times\Omega'\to\R$ be bounded and measurable.
    \begin{enumerate}
     
    \item
    The primal problem~\eqref{rotIntro} admits a unique optimizer $\pi\in \Pi( P,Q)$.
    \item 
    Strong duality holds, i.e., the value of the primal problem~\eqref{rotIntro} equals the value of the dual problem~\eqref{dualIntro}.
    
    \item
    The dual problem~\eqref{dualIntro} admits a (non-unique) optimizer $(f,g)\in L^\infty(P)\times L^\infty(Q)$.
    \item
    A pair $(f,g)\in L^\infty(P)\times L^\infty(Q)$ is an optimizer of the dual problem~\eqref{dualIntro} if and only if it satisfies the first-order condition
\begin{equation}
        \label{eq:FOC.as}
        \begin{cases}
             &\int  \psi'\bigl(\frac{f(\cdot)+ g(y)-c(\cdot,y)}{\eps}\bigr) dQ(y)=1 \quad P\text{-a.s.},\\[.1em]
             &\int \psi'\bigl(\frac{f(x)+ g(\cdot)-c(x,\cdot)}{\eps}\bigr) dP(x)=1  \quad Q\text{-a.s.}
        \end{cases}
\end{equation}    
    
    \item
    Any dual optimizer $(f, g)$ yields a primal optimizer via the density
    \begin{equation}
    \label{eq:dual-primal}
    \rho_\eps(x,y):=\frac{d\pi}{d(P\otimes Q)}(x,y) =\psi'\biggl(\frac{f(x)+g(y)-c(x,y)}{\eps}\biggr),
\end{equation}
where $\psi'$ is the derivative of  $\psi$. 
   
    \item \label{item:goodPot} Let $(f,g)\in L^\infty(P)\times L^\infty(Q)$ satisfy~\eqref{eq:FOC.as}. Then there exists a unique bounded measurable extension $(f,g)$ to $(\Omega,\Omega')$ that is a version of the given equivalence classes and satisfies
    \begin{equation}
        \label{eq:FOC.pw}
        \begin{cases}
             &\int  \psi'\bigl(\frac{f(x)+ g(y)-c(x,y)}{\eps}\bigr) dQ(y)=1 \quad \text{for all }x\in\Omega,\\[.1em]
             &\int \psi'\bigl(\frac{f(x)+ g(y)-c(x,y)}{\eps}\bigr) dP(x)=1  \quad \text{for all }y\in\Omega'.
        \end{cases}
\end{equation}
    For such a version $(f, g)$, the uniform bound $\|f\oplus g \|_\infty \leq 5 \|c\|_\infty + \eps \phi'(1) < \infty$ holds on $\Omega\times\Omega'$.
    Moreover, if $c(\cdot, y)$ 
    admits a uniform modulus of continuity $\omega$ uniformly in~$y$, then $f$ admits $\omega$ as well. Similarly for $g$. In particular, if $c\in \mathcal{C}(\Omega\times\Omega')$ admits a uniform modulus of continuity $\omega$, then $f$ and $g$ both admit~$\omega$.
    
\item If $\spt P$ is connected, $c$ is continuous, $\spt P=\Omega$ and $\spt Q=\Omega'$, then the solution of~\eqref{eq:FOC.pw} is unique up to additive constant: the solution set is $\{(f+a,g-a):a\in\R\}$, where $(f,g)$ is as in~$(vi)$.
\end{enumerate}
\end{proposition}

In the following, when referring to potentials $(f,g)$ associated with $(P,Q)$, we always choose the version/extension mentioned in~(vi), and moreover center $f$ w.r.t.\ the first marginal,
\begin{equation}\label{eq:normalization}
    \int f dP=0.
\end{equation}
This uniquely specifies the potentials
when they are determined up to additive constants.
 
The following consistency result states that
convergence of the marginals
implies convergence of
the associated dual potentials. 

\begin{lemma}[Consistency of potentials]\label{lemma:stability}  Let $P,Q$ be probability measures on $\mathbb{R}^d$ with compact supports $\Omega,\Omega'$. Assume that $\Omega$ is connected and $c\in \mathcal{C}(\Omega\times\Omega')$. Let $\{P_n\}$ and $\{Q_n\}$ be sequences of measures supported in $\Omega$ and $\Omega'$, respectively, converging in distribution to $P$ and $Q$, respectively. Let $(f_n,g_n)$ be potentials associated with $(P_n,Q_n)$, and let $(f_\eps,g_\eps)$ be the (unique in the sense of \Cref{pr:ROTprelims}~(vii)) potentials associated with $(P,Q)$. Then 
$$ \| f_n\oplus g_n- f_\eps\oplus g_\eps\|_{\infty,\Omega\times\Omega'} \to 0.
$$
\end{lemma}

Throughout the paper, we will use the following
assumptions.
\begin{standingassumption}[Bounded supports]\label{assumption:bounded-support}
The population measures $P$ and $Q$ have compact supports $\Omega$ and $\Omega'$.    
\end{standingassumption}
\begin{standingassumption}[Bounded cost]\label{assumption:bounded} The cost function $c:\Omega\times \Omega'\to \R$ is bounded and measurable.   
\end{standingassumption} 

We also fix a probability space $(\boldsymbol{\Omega}, \mathcal{A}, \mathbb{P})$ on which the i.i.d.\ samples  $X_i\sim P$ and $Y_i\sim Q$ are defined. We assume that both samples are independent.  The associated empirical measures are denoted $\widehat{P} = \frac{1}{n}\sum_{i=1}^n \delta_{X_i}$ and $\widehat{Q} = \frac{1}{n}\sum_{j=1}^n \delta_{Y_j}$. We choose and fix associated dual potentials $(\widehat f,\widehat g)$ that are $\mathcal{A}$-measurable.

\section{Statistical complexity of the cost}\label{Section:Cost}
In this section we state the main results on the statistical complexity of the cost where the bias has the slow rate $n^{-\frac{1}{2}}$. Proofs are postponed to \cref{sect:proofs-cost}.

\subsection{General cost bounds}\label{Section:Cost-general}
We start with general (if partly crude) bounds for the two terms in the bias--variance decomposition 
\begin{equation}\label{eq:bias-variance-decomp}
    \EE\bigl[\bigl(\widehat{\rm OT}_{\phi,\eps}- {\rm OT}_{\phi,\eps} \bigr)^2\bigr] = \underbrace{\bigl(\EE\bigl[\widehat{\rm OT}_{\phi,\eps}\bigr] - {\rm OT}_{\phi,\eps} \bigr)^2}_{\text{bias}}+\underbrace{ {\rm Var}\bigl[\widehat{\rm OT}_{\phi,\eps}\bigr]}_{\text{variance}} .
\end{equation}
The first result is a parametric rate for the bias term, valid for bounded measurable transport cost~$c$ and general regularization~$\phi$. 

\begin{proposition}[Bias bound]\label{pr:general} We have
$$\bigl(\EE\bigl[\widehat{\rm OT}_{\phi,\eps}\bigr]- {\rm OT}_{\phi,\eps}\bigr)^2 \leq 2\min_{a,b\in \R}\frac{\EE\Bigl[\|f_\eps-\widehat f-a\|_{L^2(\widehat{P})}^2+ \|g_\eps-\widehat g-b\|_{L^2(\widehat{Q})}^2 \Bigr] {\rm Var}_{{P\otimes Q}} [\rho_\eps]}{n},$$
where $\rho_\eps$ denotes the density~\eqref{eq:dual-primal} of the optimal plan~$\pi_\eps$ for~$(P,Q)$. As a consequence, 
$$\bigl(\EE\bigl[\widehat{\rm OT}_{\phi,\eps}\bigr]- {\rm OT}_{\phi,\eps}\bigr)^2 \leq 8\bigl(5\|c\|_\infty+\eps\phi'(1)\bigr)^2 \, \frac{{\rm Var}_{{P\otimes Q}} [\rho_\eps]}{n}.  $$
\end{proposition}

The second display is obtained from the first via the bound of \cref{pr:ROTprelims}(vi) for the potentials. As mentioned in \cref{se:proof-strategies}, the first display is significantly stronger since the difference of the potentials tends to zero, and this will be crucial to obtain the fast rate in \cref{sect:fast-bias}.

Meanwhile, using \Cref{pr:ROTprelims}(v),(vi) to bound ${\rm Var}_{{P\otimes Q}} [\rho_\eps]$ yields our first explicit bound on the sample 
complexity of regularized OT.

\begin{corollary}\label{coro:Cost-bounded}
We have 
    $$\bigl(\EE\bigl[\widehat{\rm OT}_{\phi,\eps}\bigr]- {\rm OT}_{\phi,\eps}\bigr)^2 \leq 8\bigl(5\|c\|_\infty+\eps\phi'(1)\bigr)^2 \, \frac{ \psi'\bigl(\frac{6\|c\|_\infty+\eps\phi'(1)}{\eps} \bigr)^2}{n}.  $$
\end{corollary}
 In the case of Tsallis entropies this yields
\begin{equation}
\label{eqn:tsallis_preliminary}
\bigl(\EE\bigl[\widehat{\rm OT}_{\phi_p,\eps}\bigr]- {\rm OT}_{\phi_p,\eps}\bigr)^2
\leqslant  \begin{cases} 8\bigl(5\|c\|_\infty+\eps\bigr)^2\,
\frac{(6\|c\|_\infty+\eps)^{2/(p-1)}}{\eps^{2/(p-1)}\,n}, & p > 1 \\[0.8em]
8\bigl(5\|c\|_\infty+\eps\bigr)^2\,\frac{\exp(12\|c\|_\infty/\eps)}{n} & p = 1.
\end{cases}
\end{equation}

In particular, \Cref{coro:Cost-bounded} already implies bounds with dimension-free and polynomial
dependence on $1/\eps$ for a broad class of divergences, including Tsallis entropies of any order $p \in (1, \infty)$. For EOT, however, the corresponding bound exhibits an exponential dependence on $1/\eps$. We will see in \cref{section:Lipschitz-cost} how this bound
can be improved for Lipschitz transport costs.

Next, we turn to the variance of $\widehat{\rm OT}_{\phi,\eps}$ and show that it is of order
$1/n$ under our general assumptions, with no dependence on the regularization parameter~$\eps$.

\begin{proposition}[Variance bound]\label{pr:variance} We have
    $$ {\rm Var}\bigl[\widehat{\rm OT}_{\phi,\eps}\bigr]\leq \frac{4\|c\|_\infty^2}{n}.  $$
\end{proposition}

Independence of the regularization parameter~$\eps$ is expected, as the variance bound also holds for unregularized optimal transport \cite{delBarrio.el.al.2024.AIHP}. By contrast, the bias bound in \cref{pr:general} depends on $\eps$ through the variance $\mathrm{Var}_{P \otimes Q}[\rho_\eps]$. 

Before proceeding to more refined sample complexity results, we state a central limit theorem for the regularized cost.

\begin{theorem}[Central limit theorem]\label{Theorem:CLT}
Suppose that $\Omega$ is connected and $c\in \mathcal{C}(\Omega\times\Omega')$. Consider the estimator 
$$ \widehat{W}= \widehat{\rm OT}_{\phi,\eps} - \int h_\eps(x,y) d(\widehat{P} \otimes \widehat{Q})(x,y),$$
where 
$$ h_\eps(x,y)= f_\eps(x) + g_\eps(y) - \eps \psi\biggl(\frac{f_\eps(x)+g_\eps(y)-c(x,y)}{\eps} \biggr).$$
Then $ n {\rm Var}[\widehat{W}] \to 0 $, and as a consequence, 
    $\sqrt{n}(\widehat{\rm OT}_{\phi,\eps}-\EE[\widehat{\rm OT}_{\phi,\eps}])  $ converges in distribution to a Gaussian random variable with mean zero and variance given by
\[
{\rm Var}\Biggl[f_\eps(X)+g_\eps(Y)
-\eps\Bigl(
\int \psi(\xi_\eps(X,y))\,dQ(y)
+\int \psi(\xi_\eps(x,Y))\,dP(x)
\Bigr)\Biggr],
\]
where $\xi_\eps(x,y):=(f_\eps(x)+g_\eps(y)-c(x,y))/\eps$ and $(X,Y)\sim P\otimes Q$.
\end{theorem}

\subsection{Bounds on the variance of the density for Lipschitz costs}\label{section:Lipschitz-cost}

In \Cref{coro:Cost-bounded}, $\mathrm{Var}_{P \otimes Q}[\rho_\eps]$
was bounded crudely in terms of $\|c\|_\infty$.
We now strengthen this bound when the cost
is Lipschitz, obtaining refined results for the sample complexity. We write $\mathbb{B}_R(x_0)$ for the open ball of radius $R > 0$ centered at $x_0 \in \mathbb{R}^d$, and $\overline{\mathbb{B}}_R(x_0)$ for its closure.

\begin{assumption}[Cost is Lipschitz in first
argument]\label{assumption-cost-lip-one-side}
   There exists $L>0$ such that
    $$ |c(x,y)-c(z,y)|\leq L\|x-z\| \quad\text{for all } x,z\in \Omega \text{ and }y\in \Omega'.$$
\end{assumption}

\begin{lemma}\label{lemma:Bound-lipschitz} Let \cref{assumption-cost-lip-one-side} hold. Then 
   $$ \int  \rho_\eps(x,y)^2 dP(x)dQ(y)\leq  \int  \inf_{r>0}\psi'\biggl( \phi'\biggl( \frac{1}{P(\overline{\mathbb{B}}_{r}(x))} \biggr)+  \frac{2L}{\eps} r \biggr) dP(x). $$
\end{lemma}
\begin{remark}
    \Cref{lemma:Bound-lipschitz} involves only one of the two marginals (the one referenced in \Cref{assumption-cost-lip-one-side}). If the cost function is $L$-Lipschitz in both variables, then by symmetry, the lower complexity adaptation principle holds:
 $$ \int  \bigl(\rho_\eps(x,y)\bigr)^2 dP(x)dQ(y)\leq \min_{\mu\in \{P,Q\}}\int  \inf_{r>0}\psi'\biggl( \phi'\biggl( \frac{1}{\mu(\overline{\mathbb{B}}_{r}(x))} \biggr)+  \frac{2L}{\eps} r \biggr) d\mu(x). $$
\end{remark}

Combining \cref{lemma:Bound-lipschitz} with \cref{pr:general,pr:variance} gives the following bound.

\begin{corollary}\label{corollary-Bias-variance-adaptation}
 Let \cref{assumption-cost-lip-one-side} hold. Then
 \begin{multline*}
      \EE\bigl[\bigl(\widehat{\rm OT}_{\phi,\eps}- {\rm OT}_{\phi,\eps} \bigr)^2\bigr]\\ \leq  \frac{8\bigl(5\|c\|_\infty+\eps\phi'(1)\bigr)^2}{n}   \int \inf_{r>0}\psi'\biggl( \phi'\biggl( \frac{1}{P(\overline{\mathbb{B}}_{r}(x))} \biggr)+ \frac{2L}{\eps} r \biggr)  dP(x) + \frac{4\|c\|_\infty^2}{n} .
 \end{multline*}

\end{corollary}

This abstract right-hand side will be made more concrete in the next two subsections, which will make clear that 
\Cref{corollary-Bias-variance-adaptation} is indeed a significant strengthening of \Cref{coro:Cost-bounded}.

\subsection{Semi-discrete regularized optimal transport}\label{section:Complexity-semidiscrete}
Semi-discrete transport is the setting where $P=\sum_{i=1}^N p_i \delta_{x_i}$ is supported on a set of distinct atoms $x_1, \dots, x_N \in \R^d$ (whereas $Q$ is arbitrary). Then, the unregularized OT problem does not suffer from the curse of dimensionality \cite{delBarrio.et.al.2024.BJ,Sadhu.et.al.AoAP.2024,Hundrieser.et.al.AHIP.2024}. In EOT, \cite{Stromme.24} and \cite{Groppe-Shayan-LCA-JMLR} discovered independently that
 $\EE[| \widehat{\rm EOT} - {\rm EOT}|] \lesssim n^{-\frac{1}{2}} $, where the hidden constant is independent of~$\eps$. More specifically, \cite{Groppe-Shayan-LCA-JMLR} showed this for bounded costs, and~\cite{Stromme.24} for Lipschitz costs, but the proof of 
\cite{Stromme.24} only uses Lipschitz continuity in one variable. In fact, in the semi-discrete case with $\Omega=\{x_1, \dots, x_N\}$, any bounded cost satisfies \Cref{assumption-cost-lip-one-side}: if   $N\geq 2$, for $x_i\neq x_j$ and $y\in \Omega'$,
\[
\bigl|c(x_i,y)-c(x_j,y)\bigr|\leq 2\|c\|_\infty \leq  L  \|x_i-x_j\|, \quad L:= \frac{2\|c\|_\infty}{\min_{k\neq\ell}\|x_k-x_\ell\|}.
\]
The case \(N=1\) is trivial and any \(L>0\) may be chosen.
The right-hand side of \cref{lemma:Bound-lipschitz} can then be bounded via
\[
\int \inf_{r>0}\psi'\biggl( \phi'\biggl( \frac{1}{P(\overline{\mathbb{B}}_{r}(x))} \biggr)+ \frac{2L}{\eps} r \biggr)  dP(x) \leq  \lim_{r\to 0}\sum_{i=1}^N  \psi'\biggl(    \phi'\bigl(p_i^{-1}\bigr)   +\frac{2L}{\eps} r\biggr) p_i = N.
\]
Therefore, \cref{corollary-Bias-variance-adaptation} yields the following bound, generalizing the aforementioned results for EOT to general regularizations.

\begin{corollary}[Semi-discrete]
 Assume that the support of~$P$ has cardinality $N<\infty$. Then
  $$ \EE\bigl[\bigl(\widehat{\rm OT}_{\phi,\eps}- {\rm OT}_{\phi,\eps} \bigr)^2\bigr] \leq  \frac{8\bigl(5\|c\|_\infty+\eps\phi'(1)\bigr)^2 N}{n}    + \frac{4\|c\|_\infty^2}{n} . $$
\end{corollary}

\subsection{Bounds on the variance of the density for Tsallis entropies} \label{sect:Lp-Entropies-cost-adaptation}
In this section we develop the bound
of~\Cref{lemma:Bound-lipschitz} in the case of the $L^p$ (equivalently, up to additive constants, Tsallis) regularizers, namely
$$ \phi_p(t) = \begin{cases}
    \frac{t^p - 1}{p} & \text{if } p\in (1, \infty),\\
t\log t & \text{if } p=1. 
\end{cases}$$
\begin{definition}[Covering numbers]
A \emph{$\delta$-cover} of $\Omega$ is a set of points $\{x_1, x_2, \dots, x_N\} \subset \Omega$ such that $\Omega \subset \bigcup_{i=1}^N \mathbb{B}_\delta(x_i)$. The \emph{$\delta$-covering number} $\mathcal{N}(\Omega,\delta)$ of $\Omega$ is the smallest cardinality~$N$ of a $\delta$-cover of $\Omega$.
\end{definition}
\begin{lemma}\label{lemma:Bound-energy-Lp}
Fix $\phi=\phi_p$ with $p\in[1,\infty)$ and let \cref{assumption-cost-lip-one-side} hold.
\begin{enumerate}
    \item If $p\geq 2$, 
    $$  \int  \inf_{r>0}\psi'_p\biggl( \phi'_p\biggl( \frac{1}{P(\overline{\mathbb{B}}_{r}(x))} \biggr)+  \frac{2L}{\eps} r \biggr) dP(x) \leq \inf_{r>0}
\biggl\{\mathcal{N}\Bigl(\Omega,\frac{r}{4}\Bigr)+  \Bigl( \frac{2L}{\eps} r \Bigr)^{\frac{1}{p-1}} \biggr\}. $$
    \item If $1<p< 2$,
$$  \int  \inf_{r>0}\psi'_p\biggl( \phi'_p\biggl( \frac{1}{P(\overline{\mathbb{B}}_{r}(x))} \biggr)+  \frac{2L}{\eps} r \biggr) dP(x) \leq 2^{\frac{2-p}{p-1}}  \inf_{r>0}  \biggl\{  \mathcal{N}\Bigl(\Omega,\frac{r}{4}\Bigr)+\Bigl( \frac{2L}{\eps} r \Bigr)^{\frac{1}{p-1}}\biggr\}. $$
    \item If $p=1$, 
    $$  \int  \inf_{r>0}\psi'_p\biggl( \phi'_p\biggl( \frac{1}{P(\overline{\mathbb{B}}_{r}(x))} \biggr)+  \frac{2L}{\eps} r \biggr) dP(x) \leq \inf_{r>0}    \mathcal{N}\Bigl(\Omega,\frac{r}{4}\Bigr)e^{ \frac{2L}{\eps} r } .$$
\end{enumerate}
\end{lemma}
Next, we recall the definition of Minkowski dimension. 
\begin{definition}
    The (upper) \emph{Minkowski dimension} of $\Omega$ is  
\[
{\rm dim}_M(\Omega) = \limsup_{\delta \to 0} \frac{\log \mathcal{N}(\Omega, \delta)}{-\log \delta}.
\]
(The lower Minkowski dimension is defined with $\liminf$, but since we will only use the upper dimension, for brevity, we simply call that the Minkowski dimension.)
\end{definition}

The following observation is a direct consequence of \cref{lemma:Bound-lipschitz,lemma:Bound-energy-Lp}.

\begin{lemma}\label{lemma:LCAP-for-Tsallis-bound-variance}
 Fix $\phi=\phi_p$ with $p\in[1,\infty)$ and let \cref{assumption-cost-lip-one-side} hold. For any real number $d_P> {\rm dim}_M(\Omega)$, there exist $C>0$ and $\eps_0>0$ such that for every $\eps\leq \eps_0$,
    $$ {\rm Var}_{P\otimes Q}[\rho_\eps] \leq   C 
       \eps^{-\frac{d_P}{1+d_P(p-1)}} . $$   
\end{lemma}

\begin{remark}\label{rk:equality-for-M-dim}
The assumption $d_P> {\rm dim}_M(\Omega)$ ensures that $\mathcal{N}(\Omega,\delta)\leq C'\delta^{- d_P}$, for $\delta$ small enough and some $C'>0$. In fact, the latter inequality is a sharper sufficient condition for the result. In most cases of interest, for instance when $\Omega$ is a Lipschitz submanifold, this sharper condition is also satisfied for $d_P={\rm dim}_M(\Omega)$.
\end{remark}

 The following example shows that our bound on ${\rm Var}_{P\otimes Q}[\rho_\eps]$ in \cref{lemma:LCAP-for-Tsallis-bound-variance} is sharp. 
 
\begin{example}\label{example:sharp}
    Let $P=Q$ be the uniform measure on $[0,1]^d$ and $$c(x,y)=d_{\mathbb{T}^d}(x,y)= \inf_{z\in \Z^d}\|x-y-z\|.$$ Then there exist $C>0$ and $\eps_0>0$ such that for every $\eps\leq \eps_0$, 
    $$  {\rm Var}_{P\otimes Q}[\rho_\eps] \geq  \frac{1}{C} \cdot \begin{cases}
      \eps^{-\frac{d}{  1+(p-1)d}} & \text{if $p\in (1, \infty)$},\\
         \eps^{-d} &\text{if } p=1.
    \end{cases}  $$
\end{example} 
Finally, we combine \cref{pr:general,pr:variance,lemma:LCAP-for-Tsallis-bound-variance} to obtain the main result for the general $p$-Tsallis family.

\begin{corollary}\label{corollary-bias-minkowski}
 Fix $\phi=\phi_p$ with $p\in[1,\infty)$ and let \cref{assumption-cost-lip-one-side} hold. For any $d_P> {\rm dim}_M(\Omega)$, there exist $C>0$ and $\eps_0>0$ such that  for all $\eps\leq \eps_0$,
        $$ \EE\bigl[\bigl(\widehat{\rm OT}_{\phi_p,\eps}- {\rm OT}_{\phi_p,\eps} \bigr)^2\bigr] \leq   C \frac{\eps^{-\frac{d_P}{1+d_P(p-1)}} }{n} . $$
 \end{corollary}

 \Cref{rk:equality-for-M-dim} applies to \cref{corollary-bias-minkowski} as well, and this is the form we have used in the Introduction.

We observe that \cref{corollary-bias-minkowski} indicates a much stronger dependence on the (intrinsic) dimension for the entropic regularization than the Tsallis entropies for $p > 1$. The next subsection explains this phenomenon.

 \subsection{Trade-off between regularization bias and statistical complexity}\label{section:tradeoff}
 
 We start by recalling a different notion of dimension. Denote by $\mathcal{W}_2( P,Q)$ the 2-Wasserstein distance between $P$ and $Q$.
 
\begin{definition}[Quantization dimension] \label{definition:quantization-dimension}
 We say that $P$ has \emph{quantization dimension} smaller than $D_P>0$ if there exist $C>0$ and a sequence of probability measures $\{P_n\}_{n}$, with ${\rm supp}(P_n)$ having cardinality at most~$n$, such that $ \mathcal{W}_2(P_n, P) \leq C \, n^{-1/D_P}$ for all $n$ large enough. We say that $P$ has \emph{uniform quantization dimension} smaller than $D_P>0$ if additionally the sequence $\{P_n\}_{n}$ can be chosen such that 
 \begin{equation}
     \label{uniform-quant}
     \liminf_n  \frac{\min_{x\in {\rm supp}(P_n)}P_n(x)}{\max_{x\in {\rm supp}(P_n)}P_n(x)} >0 .
 \end{equation}
\end{definition}

In this subsection we assume that the transport cost function is Lipschitz.

\begin{assumption}[Lipschitz cost]\label{assumption:lipschitz-cost}
    The cost $c$ is Lipschitz on $\Omega\times \Omega'$ with constant $L>0$.
\end{assumption}
Theorem~3.3 and Example~3.4 in \cite{EcksteinNutz.22} show that, if  $P$ has quantization dimension smaller than $D_P$ and $p\in [1,2]$, there exists a constant $C_p$  depending on $\Omega$, $p$, $d$ and the Lipschitz constant of the cost such that for $\eps\in (0,1]$,
    \begin{equation}
        \label{bound:marcel-Stephan-adaptation}
        \bigl|{\rm OT}_{\phi_p,\eps} -{\rm OT}\bigr| \leq  \begin{cases}
       C_p \,\eps^{\frac{1}{(p-1)D_P+1}} & \text{for } p>1,\\
         C_1 \,\eps |\log(\eps)+1|  &\text{for } p=1.
    \end{cases}
    \end{equation}
For $p>2$, the bound \eqref{bound:marcel-Stephan-adaptation} holds if $P$ has uniform quantization dimension smaller than $D_P$.\footnote{In \cite{EcksteinNutz.22} it is shown that  \eqref{bound:marcel-Stephan-adaptation} holds for $p>2$ if the \emph{empirical quantization dimension} is smaller than $D_P$. The empirical quantization dimension requires $P_n$ in \cref{definition:quantization-dimension} to give the same mass to each atom of its support. However, the proof of \cite{EcksteinNutz.22} (based on Lemma~2.3(ii), ibid) adapts easily to the notion of uniform quantization dimension. Indeed, if ${\rm supp}(P_n)=\{x_1, \dots, x_{n'}\}$ with $n'\leq n$ and \eqref{uniform-quant} holds, then $\frac{1}{C\, n'}\leq P_n(x_i)\leq \frac{C}{n'}$ for some constant $C$. This implies that for every probability measure $\mu$ and $ \pi\in \Pi(\mu,P_n) $, 
\begin{equation}
    \label{eq:to-imitate-proof}
    \int \biggl(\frac{d\pi}{d(\mu\otimes P_n)}\biggr)^p d(\mu\otimes P_n) \leq C^p (n')^{p-1} \leq C^p n^{p-1},
\end{equation}
so that the proof of \cite[Theorem~3.3]{EcksteinNutz.22} can be imitated using \eqref{eq:to-imitate-proof} instead of Lemma~2.3, ibid.  
}

The quantization dimension has been studied in several works (see, e.g.,  \cite{Graf.2000.quant} and the references therein). We provide an elementary proof  
showing the relation between the quantization dimension and the Minkowski dimension.
\begin{lemma}\label{lemma:Minkowski-quantization} The following hold:
\begin{enumerate}
    \item For any $d_P> {\rm dim}_M(\Omega)$, 
 $P$ has quantization dimension  $\leq d_P$. 
  \item If $\mathcal{N}(\Omega,\delta)\leq C'\delta^{- {\rm dim}_M(\Omega)}$, for $\delta$ small enough and some $C'>0$, then $P$ has quantization dimension  $\leq {\rm dim}_M(\Omega)$. 
 \item  If $d_P> {\rm dim}_M(\Omega)$ and there exist $C,r_0>0$ such that, for every $r\leq r_0$, 
 \begin{equation}
     \label{eq:Bound-uniform-quant-lemma}
     P(\mathbb{B}_r(x))\leq C\cdot P(\mathbb{B}_{\frac{r}{2}}(y)) \quad \text{for all }x,y\in \Omega,
 \end{equation}
then $P$ has uniform quantization dimension $\leq d_P$. 
 \item If \eqref{eq:Bound-uniform-quant-lemma} holds and  $\mathcal{N}(\Omega,\delta)\leq C'\delta^{- {\rm dim}_M(\Omega)}$, for $\delta$ small enough and some $C'>0$, then $P$ has uniform quantization dimension $\leq {\rm dim}_M(\Omega)$. 
\end{enumerate}
  
\end{lemma}
The following is a direct consequence of \eqref{bound:marcel-Stephan-adaptation} and \cref{lemma:Minkowski-quantization}.

\begin{lemma}[Regularization error]\label{lemma:EcksteinNutz}
Let \cref{assumption:lipschitz-cost} hold, let $d_P> {\rm dim}_M(\Omega)$, and fix $\phi=\phi_p$ with $p\in [1, \infty)$.  The following hold for $\eps\in (0,1]$:
\begin{enumerate}
    \item If  $p\in [1,2]$,
     \begin{equation*}
        \bigl|{\rm OT}_{\phi_p,\eps} -{\rm OT}\bigr| \lesssim \begin{cases}
        \eps^{\frac{1}{(p-1)d_P+1}} & \text{for } p>1,\\
         -\eps \log(\eps) &\text{for } p=1.
    \end{cases}
    \end{equation*}
    \item If $p>2$ and \eqref{eq:Bound-uniform-quant-lemma} holds, 
    $$ \bigl|{\rm OT}_{\phi_p,\eps} -{\rm OT}\bigr|  \lesssim  \eps^{\frac{1}{(p-1)d_P+1}} . $$
\end{enumerate}

\end{lemma}

We observe that for EOT, the regularization error
is independent
of $d_P$ up to multiplicative constants, while for $p > 1$ the dependence appears in the exponent of the regularization parameter. This also explains why the influence of the dimension on EOT is much more pronounced in \Cref{corollary-bias-minkowski}. Ultimately, it reflects a trade-off between the statistical complexity of the estimator and the bias introduced by the regularization.

\begin{corollary}
\label{corollary-tsallis-estimation-error}
Let \cref{assumption:lipschitz-cost} hold, let $d_P> {\rm dim}_M(\Omega)$, and fix $\phi=\phi_p$ with $p\in [1, \infty)$. The following hold for $\eps\in (0,1]$:
\begin{enumerate}
    \item If  $p\in [1,2]$,
     \begin{equation*}
        \bigl|\EE\bigl[\widehat{\rm OT}_{\phi_p,\eps}\bigr] -{\rm OT}\bigr| \lesssim \begin{cases}
        \frac{\eps^{-\frac{d_P }{  2(1+(p-1) d_P)}}}{\sqrt{n}}+\eps^{\frac{1}{(p-1)d_P+1}} & \text{for } p>1,\\
         \frac{\eps^{-\frac{d_P }{2}}}{\sqrt{n}}-\eps \log(\eps)  &\text{for } p=1.
    \end{cases}
    \end{equation*}
    \item If $p>2$ and \eqref{eq:Bound-uniform-quant-lemma} holds, 
    $$ \bigl|\EE\bigl[\widehat{\rm OT}_{\phi_p,\eps}\bigr] -{\rm OT}\bigr|  \lesssim \frac{\eps^{-\frac{d_P }{  2(1+(p-1) d_P)}}}{\sqrt{n}}+ \eps^{\frac{1}{(p-1)d_P+1}} . $$
\end{enumerate}
\end{corollary}

\Cref{rk:equality-for-M-dim} applies to both \cref{lemma:EcksteinNutz,corollary-tsallis-estimation-error}.

Next, we tune $\eps$ as a function of $n$ to optimize the rate.

\begin{remark}
    For the case of QOT (i.e., $p=2$), the optimal choice of $
    \eps$ in~\Cref{corollary-tsallis-estimation-error} to estimate OT is $\eps \asymp n^{-\frac{1+{d_P}}{{d_P}+2}},$ which yields the rate 
    $$ \bigl|\EE\bigl[\widehat{\rm OT}_{\phi_2,\eps}\bigr] -{\rm OT}\bigr| \lesssim n^{-\frac{1}{{d_P}+2}} . $$
    For EOT the optimal choice is $\eps
\asymp n^{-\frac{1}{{d_P}+2}}\,
(\log n)^{-\frac{2}{{d_P}+2}}$, which gives 
$$\bigl|\EE\bigl[\widehat{\rm OT}_{\phi_1,\eps}\bigr] -{\rm OT}\bigr| \lesssim n^{-\frac{1}{{d_P}+2}}\,
(\log n)^{\frac{{d_P}}{{d_P}+2}}. $$
\end{remark}

In the case of QOT with quadratic transport cost, the $L^2$ structure of the penalty allows the regularization error to control the variance of the optimal density, and the quadratic cost yields a sharper approximation estimate. Combining these two facts, we can further improve our bound as follows, now with the ambient dimension.

\begin{proposition}[Improved
bound for QOT with quadratic cost]
\label{prop:improved_qot_quadratic_cost}
Let
$\phi = \phi_2$ and 
$c(x, y) = \|x-y\|^2$.
Then for $\eps \in (0, 1]$,
$$
\bigl|\EE\bigl[\widehat{\rm OT}_{\phi_2,\eps}\bigr] -{\rm OT}\bigr|
\lesssim \frac{\eps^{-\frac{d}{2d + 4}}}{\sqrt{n}}
+ \eps^{\frac{2}{d + 2}}.
$$
The optimal choice of $\eps$
is $\eps \asymp n^{-\frac{d +2}{d+ 4}}$,
yielding
$$
\bigl|\EE\bigl[\widehat{\rm OT}_{\phi_2,\eps}\bigr] -{\rm OT}\bigr|
\lesssim n^{-\frac{2}{d + 4}}.
$$
\end{proposition}
We remark that when $d >4$, the minimax rate for estimating
the OT cost under the assumptions
of \cref{prop:improved_qot_quadratic_cost}
is $n^{-2/d}$, up to logarithmic factors~\cite{niles2022estimation}.
\Cref{prop:improved_qot_quadratic_cost}
therefore shows that, properly tuned, QOT nearly achieves this
optimal rate.
To the best of our knowledge,
the rate $n^{-2/(d + 4)}$ is the best known
bound for any regularized OT estimator
coming from a non-asymptotic bound for general
$\eps \in (0, 1]$.

 \section{Sample complexity of the potentials and improved bias bound}\label{Section-compl-fast}
 
 This section presents our results on the convergence
 of the dual potentials (\cref{sect-complexity-pot}) and on the fast rate for the bias of the cost (\cref{sect:fast-bias}). All proofs are deferred to \Cref{sect-proof-pot}.
 We require the following assumptions.
 
 \begin{assumption}[Strengthened assumption on the divergence]\label{assumption:divergence-strong}
    In addition to \cref{assumption:divergence}, the function $\psi$ belongs to $\mathcal{C}^{1,1}_{loc}(\R)$, with $\psi''\in L^{\infty}_{loc}(\R)$ non-decreasing on its domain. We fix a monotone and measurable representative of the a.e.~second derivative $\psi''$. 
\end{assumption}
In the case of Tsallis entropies, this
assumption means that $p \in [1,2]$.

We also assume the following for the first marginal~$P$.

\begin{assumption}[Strengthened assumption on $P$]\label{assumption-P-strong}
   The support~$\Omega$ of the first population marginal~$P$ is bounded and convex, and $P$ admits a Lebesgue density $\rho$ that is bounded from above and below:
   $$0<\lambda_P\leq  \rho (x) \leq \Lambda_P<\infty , \quad x\in \Omega.$$
\end{assumption}

\begin{remark}
Convexity of $\Omega$ can be relaxed to a connected domain with Lipschitz boundary, at the expense of an additional constant, by following the arguments in \cite[Appendix~B]{GonzalezSanzNutzRiveros.26}.
\end{remark} 

\subsection{Sample complexity of the potentials}\label{sect-complexity-pot}

We begin with a preliminary remark to introduce a notation that will be needed to state the constant in the main result.

\begin{remark}
\label{remark:deterministic-control} Since  $\psi''$ is locally bounded and $ \widehat{f} \oplus \widehat{g}$ and $c$ are bounded by a deterministic constant, 
there exists a   function $\gamma:\NN \times (0,\infty) \to (0, \infty)$ such that for every $n\in \NN$ and $\eps\in (0,\infty),$ 
\begin{equation}
    \label{eq:gamma}
    \mathbb{P}\biggl( \max_j\frac{1}{n}\sum_{i} \pi_{i,j}^{(1)}   \leq \gamma(n,\eps)\biggr)=1 \quad  \text{and}\quad \mathbb{P}\biggl( \max_i\frac{1}{n}\sum_{j} \pi_{i,j}^{(1)}   \leq \gamma(n,\eps)\biggr)=1 ,
\end{equation}
where
    $$ \pi_{i,j}^{(1)}=\psi''\biggl(\frac{\widehat{f}(X_i)+\widehat{g}(Y_j)-c(X_i,Y_j)}{\eps} +\frac{\delta_\psi}{4} \biggr). $$
\end{remark}

\begin{theorem}\label{theorem:potentials}
    Let \cref{assumption:lipschitz-cost,assumption:divergence-strong,assumption-P-strong} hold. Let $(\delta_\psi,t_\psi)$ be a pair as in \Cref{assumption:divergence} such that $\psi''(t_\psi-\delta_\psi)>0$, and let $\gamma(n,\eps)$ be such that \eqref{eq:gamma} holds. For every fixed $u>\max(d,2)$, there exist constants $C$  and $n_0\in\NN$ such that for all $n\ge n_0$,
    and $\eps\in (0,1]$, \begin{align*}
 \EE\biggl[ \int & \Bigl(\bigl( \widehat{f} - f_\eps \bigr) \oplus \bigl( \widehat{g} - g_\eps \bigr) \Bigr)^2  d(\widehat{P}\otimes \widehat{Q})\biggr]
\leqslant \frac{2^{8} \max\left(\frac{2^{6}\bigl(5\|c\|_{\infty} + \eps \phi'(1)\bigr)^2}{\delta_\psi^2}, 1\right) }{ (\beta_{n,\eps})^2  }  \frac{{\rm Var}_{{P\otimes Q}} [\rho_\eps]}{n} \\
 &+  \frac{ 8\bigl(5 \|c\|_\infty + \eps \phi'(1)\bigr)^2}{n}  \biggl(\frac{1}{\eps^u} + \frac{C}{\Bigl(\min_{x\in \Omega}  P\Bigl(\mathbb{B}_{\frac{\delta_\psi \eps}{8L}} (x)\Bigr) \Bigr)^2 } + \frac{C}{\Bigl(\min_{y\in \Omega'}  {Q}\Bigl(\mathbb{B}_{\frac{\delta_\psi \eps}{8L}} (y)\Bigr) \Bigr)^2 }  \biggr),
    \end{align*}
    where
    $$  \beta_{n,\eps}= \min\biggl(\frac{ 4{\alpha}_{n,\eps} \psi''\bigl(t_\psi-\delta_\psi\bigr) \min_{x\in \Omega}  P\Bigl(\mathbb{B}_{\frac{\delta_\psi \eps}{8L}} (x)\Bigr)}{ \psi''\bigl(t_\psi-\delta_\psi\bigr) \min_{x\in \Omega}  P\Bigl(\mathbb{B}_{\frac{\delta_\psi \eps}{8L}} (x)\Bigr)+ 8\gamma(n,\eps)}, \frac{1}{2}\psi''\bigl(t_\psi-\delta_\psi\bigr) \min_{x\in \Omega}  P\Bigl(\mathbb{B}_{\frac{\delta_\psi \eps}{8L}} (x)\Bigr) \biggr),  $$
and
$$ \alpha_{n,\eps}=   \frac{ \lambda_P^2 \psi''\bigl(t_\psi -\delta_\psi  \bigr)^2  \min_{y\in \Omega' }  {Q}\Bigl(\mathbb{B}_{\frac{\delta_\psi \eps}{8L}}(y)\Bigr)}{8\Lambda_P^2 \gamma(n,\eps) \Bigl\lceil\frac{16 L{\rm diam}(\Omega)}{\delta_\psi\eps } \Bigr\rceil^{d+2}}. $$
\end{theorem}
\begin{remark}[On the choice of $(\delta_\psi,t_\psi)$]  In the case of $\phi=\phi_p$ for $p\in (1,2]$ we can choose 
$ t_{\psi_p}=1 $ and $ \delta_{\psi_p}=1/2 $. Then
$$ \psi''_p\bigl(t_\psi-\delta_\psi\bigr) =  {(q-1)2^{2-q}} , \quad \text{where } q= \frac{p}{p-1}. $$
    
\end{remark}
     Note that for any of the divergences, the inequality
 $$ \pi_{i,j}^{(1)}\leq \psi''\biggl(\frac{6 \|c\|_\infty + \eps \phi'(1)}{\eps} +\frac{\delta_\psi}{4} \biggr) $$
 allows us to find $\gamma$ constant in $n$. In the case of Tsallis entropies, this can be improved as follows.

\begin{lemma}\label{lemma:Bound-Gamma}
Fix $\phi=\phi_p$ for $p\in [1,2]$. 
In the setting of \cref{theorem:potentials} and with the choice $ t_{\psi_p}=1 $ and $ \delta_{\psi_p}=1/2 $, there exists a finite constant $C$ (depending on $p$ but independent of $\eps$ and $n$) such that for every $n\in \NN$, 
$$  \mathbb{P}\biggl( \max_j\frac{1}{n}\sum_{i} \pi_{i,j}^{(1)}   \leq C\biggr)=1, \qquad  \mathbb{P}\biggl( \max_i\frac{1}{n}\sum_{j} \pi_{i,j}^{(1)}   \leq C\biggr)=1. $$ 
\end{lemma}

 \cref{lemma:Bound-Gamma} implies that if $P$ satisfies \Cref{assumption-P-strong} and $Q$ satisfies the analogous density assumption on $\Omega'$, then
 $\alpha_{n,\eps} \gtrsim \eps^{2 d+2}$ and $\beta_{n,\eps} \gtrsim  \eps^{3 d+2} $.  Choosing a fixed $u>\max(d,2)$ sufficiently close to $\max(d,2)$, the second term in \cref{theorem:potentials} is dominated for small $\eps$. Hence, for $p\in [1,2]$,  
 $$ \EE\biggl[ \int  \Bigl(\bigl( \widehat{f} - f_\eps \bigr) \oplus \bigl( \widehat{g} - g_\eps \bigr) \Bigr)^2  d(\widehat{P}\otimes \widehat{Q})\biggr] \lesssim \frac{1+ {\rm Var}[\rho_\eps]}{n\, \eps^{6 d+4} } \lesssim \frac{ 1}{n\, \eps^{6 d+4 +\frac{d}{1+d(p-1)}} },  $$
 where we used \cref{lemma:LCAP-for-Tsallis-bound-variance} and \cref{rk:equality-for-M-dim} in the last inequality. This shows \Cref{Theorem:potentials-intro} in the Introduction. In particular, specializing to EOT ($p = 1$) yields
 $$\EE\biggl[ \int  \Bigl(\bigl( \widehat{f} - f_\eps \bigr) \oplus \bigl( \widehat{g} - g_\eps \bigr) \Bigr)^2  d(\widehat{P}\otimes \widehat{Q})\biggr] \lesssim \frac{1}{n\,  \eps^{7d+4}},$$
improving the rate obtained in \cite[Lemma~27]{Stromme.24} for $d$-dimensional manifolds without boundary. 
\subsection{Fast bias bound}\label{sect:fast-bias}

Returning to the bias of the empirical cost, we apply \cref{theorem:potentials} to bound the right-hand side in our general bias estimate, \cref{pr:general}. This yields the following fast rate for the bias; for simplicity we do not track all the constants. 

\begin{theorem}[Fast rate bias bound]\label{theorem:fast-bias-general}
    Let \cref{assumption:lipschitz-cost,assumption:divergence-strong,assumption-P-strong} hold. Let $(\delta_\psi,t_\psi)$ be as in \cref{assumption:divergence} with $\psi''(t_\psi-\delta_\psi)>0$.  Let $\gamma(n,\eps)$ satisfy \eqref{eq:gamma} and define $\beta_{n,\eps}$ as in \cref{theorem:potentials}. For every fixed ${u>\max(d,2)}$  there exists $C>0$ and $n_0\in\NN$ such that, for all $n\ge n_0$ and $\eps\in (0,1]$, 
    \begin{align*} &\bigl(\EE\bigl[\widehat{\rm OT}_{\phi,\eps}\bigr]- {\rm OT}_{\phi,\eps}\bigr)^2 \\&\lesssim   \frac{  \bigl({\rm Var}_{{P\otimes Q} } [\rho_\eps] \bigr)^2 }{ (\beta_{n,\eps})^2 n^2 } 
    +  \frac{{\rm Var}_{{P\otimes Q} } [\rho_\eps]}{n^2} \biggl(\frac{1}{\eps^u} + \frac{C}{\Bigl(\min_{x\in \Omega}  P\Bigl(\mathbb{B}_{\frac{\delta_\psi \eps}{8L}} (x)\Bigr) \Bigr)^2 } + \frac{C}{\Bigl(\min_{y\in \Omega'}  {Q}\Bigl(\mathbb{B}_{\frac{\delta_\psi \eps}{8L}} (y)\Bigr) \Bigr)^2 }  \biggr) 
    \end{align*}
with $\beta_{n,\eps}$ as in \cref{theorem:potentials}.

As a consequence, if ${d\ge 2}$ and 
$Q$ satisfies the analogue of \cref{assumption-P-strong} and $\gamma(n,\eps)$ is bounded uniformly in $n$ and $\eps$, we can choose ${u\in(\max(d,2),2d)}$ to obtain, for  $n\ge n_0$ and $\eps\in (0,1]$,
$$ \bigl(\EE\bigl[\widehat{\rm OT}_{\phi,\eps}\bigr]- {\rm OT}_{\phi,\eps}\bigr)^2 
     \lesssim    \frac{   \eps^{-(6d+4)} \bigl({\rm Var}_{{P\otimes Q} } [\rho_\eps] \bigr)^2 }{  n^2 }  
+\frac{\eps^{-2d} {\rm Var}_{{P\otimes Q} } [\rho_\eps]}{n^2 }.  $$
\end{theorem}
\Cref{theorem:fast-bias-general} applies in particular to the Tsallis entropies with $p\in [1,2]$.  Then specializing to that case, choosing $\gamma(n,\eps)$ as in \cref{lemma:Bound-Gamma} and assuming that $Q$ satisfies the analogue of \cref{assumption-P-strong},   yields $$ \bigl|\EE\bigl[\widehat{\rm OT}_{\phi_p,\eps}\bigr]- {\rm OT}_{\phi_p,\eps}\bigr| 
     \lesssim     \frac{ \eps^{-(3 d+2 +\frac{d}{1+d(p-1)})}}{n },$$
which is \cref{Theorem:fast-bias} in the Introduction.

\bibliography{Ref2.bib}
\bibliographystyle{abbrv}
\appendix
\section{Proofs of \Cref{section:Notation-background}}

\begin{proof}[Proof of \cref{pr:ROTprelims}]
This follows from standard arguments as in the proofs of \cite[Proposition~2.3, Theorem~3.2]{gonzalezsanz.2025.sparseregularizedoptimaltransport} (which are in turn partially based on \cite{BayraktarEckstein.2025.BJ}). Note that while~\cite{gonzalezsanz.2025.sparseregularizedoptimaltransport} assumes $\psi\in \mathcal{C}^2(\R)$, that stronger assumption is not needed here---it is only used for the bound in \cite[Proposition~2.3]{gonzalezsanz.2025.sparseregularizedoptimaltransport}(vii)  that we do not claim here. 
\end{proof}

\begin{proof}[Proof of \cref{lemma:stability}]
In view of the uniform continuity and boundedness that follow from \cref{pr:ROTprelims}(vi) and the centering~\eqref{eq:normalization}, the Arzel\`a--Ascoli theorem allows us to choose uniformly convergent subsequences $f_n\to f_\infty$ and $g_n\to g_\infty$. Consider the dual objective
$$
\Phi_{\phi,\eps}^{n}(f,g ) =\int\Bigl\{ f(x) + g(y) - \eps\cdot\psi\biggl(\frac{f(x)+g(y)-c(x,y)}{\eps} \biggr) \Bigr\}d(  P_n \otimes   Q_n )(x,y),$$
and similarly $\Phi_{\phi,\eps}$ for $(P,Q)$. As $P_n\otimes Q_n$ converges in distribution to $P\otimes Q$, we have both
$$
\Phi_{\phi,\eps}^{n}(f_n,g_n)\to \Phi_{\phi,\eps}(f_\infty,g_\infty)
\quad\mbox{and}\quad 
\Phi_{\phi,\eps}^{n}(f_\eps,g_\eps)\to \Phi_{\phi,\eps}(f_\eps,g_\eps).
$$
On the other hand, the definition of $(f_n,g_n)$ yields that 
$$
    \Phi_{\phi,\eps}^{n}(f_n,g_n )\geq \Phi_{\phi,\eps}^{n}(f_\eps,g_\eps).
$$
It follows that $(f_\infty,g_\infty)$ is optimal for $\Phi_{\phi,\eps}$. Hence, $(f_\infty,g_\infty)=(f_\eps,g_\eps)$ by the uniqueness in \cref{pr:ROTprelims}(vii) and the normalization~\eqref{eq:normalization}, and moreover the whole sequence $(f_n,g_n)$ must converge to this limit.
\end{proof}

\section{Proofs of \cref{Section:Cost}}\label{sect:proofs-cost} 
 \begin{proof}[Proof of \Cref{pr:general}] 
As $$ \EE\bigl[ \widehat{\Phi}_{\phi,\eps}({\widehat{f}}, \widehat{g})\bigr]\geq \EE\bigl[\widehat{\Phi}_{\phi,\eps}({f}_\eps, g_\eps)\bigr]={\rm OT}_{\phi,\eps},$$ we just need to upper bound 
$$ 0\leq  \EE\bigl[ \widehat{\Phi}_{\phi,\eps}({\widehat{f}}, \widehat{g}) - \widehat{\Phi}_{\phi,\eps}({f}_\eps, g_\eps) \bigr].$$
Using the concavity of the integrand, we get 
\begin{align*}
\EE\bigl( \widehat{\Phi}_{\phi,\eps}({\widehat{f}}, \widehat{g}) - \widehat{\Phi}_{\phi,\eps}({f}_\eps, g_\eps) \bigr)&\;\leq\;
\EE\biggl[ \int \bigl[\bigl( \widehat{f} - f_\eps \bigr) \oplus \bigl( \widehat{g} - g_\eps \bigr) \bigr]
\bigl( 1 - \psi'\biggl(\frac{f_\eps \oplus g_\eps - c}{\eps}\biggr) \bigr) \, d\widehat{P} \otimes \widehat{Q} \biggr]\\
&
= \EE\biggl[\int \bigl( \widehat{f} - f_\eps \bigr)  \int 
\bigl( 1 - \psi'\biggl(\frac{f_\eps \oplus g_\eps - c}{\eps}\biggr) \bigr) \, d\widehat{Q}  d\widehat{P} \biggr]\\
&\qquad + \EE \biggl[\int \bigl( \widehat{g} - g_\eps \bigr)  \int 
\bigl( 1 - \psi'\biggl(\frac{f_\eps \oplus g_\eps - c}{\eps}\biggr) \bigr) \, d\widehat{P}  d\widehat{Q} \biggr].
\end{align*}
Set  $ \xi:= 1 - \psi'\big(\frac{f_\eps \oplus g_\eps - c}{\eps}\big)$; 
 we upper bound $  \EE\big[\int ( \widehat{f} - f_\eps )  \int 
\xi \, d\widehat{Q} d\widehat{P} \big] $. First note that 
since
$$  \EE\biggl[\iint
\xi \, d\widehat{Q} d\widehat{P} \biggr] =0,$$
exchangeability of the samples and the Cauchy–Schwarz inequality yield for every $a\in \R$ that
\begin{align*}
   \EE\biggl[\int ( \widehat{f} - f_\eps )  \int 
\xi \, d\widehat{Q} d\widehat{P} \biggr]&
=\EE\biggl[\int ( \widehat{f} - f_\eps -a)  \int 
\xi \, d\widehat{Q} d\widehat{P} \biggr]\\&= \frac{1}{n} \sum_{i=1}^n  \EE \biggl[\bigl( \widehat{f}(X_i) - f_\eps(X_i) -a\bigr) \biggl( \int 
\xi(X_i,y) \, d\widehat{Q}(y) \biggr)\biggr]\\
&=  \EE \biggl[\bigl( \widehat{f}(X_1) - f_\eps(X_1) -a\bigr) \biggl( \int 
\xi(X_1,y) \, d\widehat{Q}(y) \biggr)\biggr]\\
&\leq \sqrt{\EE \bigl[\bigl( \widehat{f}(X_1) - f_\eps(X_1) -a\bigr)^2 \bigr]\EE\biggl[\biggl( \int 
\xi(X_1,y) \, d\widehat{Q}(y) \biggr)^2\biggr]}.
\end{align*}
The first-order conditions imply $\EE[\xi(X_1,Y_j)\mid X_1]=0$ and hence $\EE[\xi(X_1,Y_j) \xi(X_1,Y_k)]=0$ for $j\neq k$. Thus
$$ \EE\biggl[\biggl( \int 
\xi(X_1,y) \, d\widehat{Q}(y) \biggr)^2 \biggr] = \frac{1}{n^2}\sum_{j,k=1}^n  \EE\bigl[\xi(X_1,Y_j) \xi(X_1,Y_k)\bigr]= \frac{1}{n}\EE\bigl[\bigl(\xi(X_1,Y_1)\bigr)^2 \bigr], $$
from which we derive that
$$  \EE\biggl[\int ( \widehat{f} - f_\eps )  \int 
\xi \, d\widehat{Q} d\widehat{P} \biggr]\leq   \Biggl(\frac{\EE\bigl[\|f_\eps-\widehat f-a\|_{L^2(\widehat{P})}^2\bigr] \EE\bigl[\bigl(\xi(X_1,Y_1)\bigr)^2 \bigr] }{n} \Biggr)^{\frac{1}{2}}$$
for every $a\in \R$. A symmetric argument shows $$ \EE \biggl[\int \bigl( \widehat{g} - g_\eps \bigr)  \int 
\xi \, d\widehat{P}  d\widehat{Q} \biggr]\leq  \Biggl(\frac{\EE\bigl[\|g_\eps-\widehat g-b\|_{L^2(\widehat{Q})}^2\bigr] \EE\bigl[\bigl(\xi(X_1,Y_1)\bigr)^2 \bigr] }{n} \Biggr)^{\frac{1}{2}}$$
for all $b\in \R$, and we conclude the proof by an 
application of Young's inequality.
\end{proof}

 \begin{proof}[Proof of \Cref{pr:variance}]
Call $\widehat{P}^{(i)}$ (resp.~$\widehat{Q}_{(j)}$) the empirical measure where $X_i$ (resp.~$Y_j$) has been replaced by an independent copy $X_i' $ (resp.~$Y_j'$). Let   $\widehat{\rm OT}_{\phi,\eps}^{(i)}$ and $(\widehat{f}^{(i)},\widehat{g}^{(i)} )$ be  the  ROT cost and potentials for the pair $(\widehat{P}^{(i)},\widehat{Q})$. Similarly, $(\widehat{\rm OT}_{\phi,\eps})_{(j)}$ and $(\widehat{f}_{(j)},\widehat{g}_{(j)} )$ are the cost and potentials for $(\widehat{P},\widehat{Q}_{(j)})$.  The Efron--Stein inequality \cite[Section~3.1]{Boucheron.el.al.2013.book} yields 
\begin{align*}
    {\rm Var}\bigl[\widehat{\rm OT}_{\phi,\eps}\bigr] &\leq \frac{1}{2}\sum_{i=1}^n  \EE\bigl[ \bigl(\widehat{\rm OT}_{\phi,\eps}^{(i)}-\widehat{\rm OT}_{\phi,\eps}\bigr)^2\bigr]+ \frac{1}{2}\sum_{j=1}^n  \EE\bigl[ \bigl((\widehat{\rm OT}_{\phi,\eps})_{(j)}-\widehat{\rm OT}_{\phi,\eps}\bigr)^2\bigr]\\
    &= \frac{n}{2} \bigl(  \underbrace{\EE\bigl[ \bigl(\widehat{\rm OT}_{\phi,\eps}^{(1)}-\widehat{\rm OT}_{\phi,\eps}\bigr)^2\bigr]}_{=:A} + \underbrace{\EE\bigl[ \bigl((\widehat{\rm OT}_{\phi,\eps})_{(1)}-\widehat{\rm OT}_{\phi,\eps}\bigr)^2\bigr]}_{=:B} \bigr),
\end{align*}
    where the equality follows from the exchangeability of the sample. We bound the term denoted~$A$; the
    proof for $B$ is analogous. By optimality, 
    \begin{align*}
        \widehat{\Phi}_{\phi,\eps}& 
        ({\widehat{f}}, \widehat{g}) -\widehat{\Phi}_{\phi,\eps}^{(1)} ({\widehat{f}}^{(1)}, {\widehat{g}}^{(1)})
        \leq \widehat{\Phi}_{\phi,\eps}({\widehat{f}}, \widehat{g}) -\widehat{\Phi}_{\phi,\eps}^{(1)} ({\widehat{f}}, \widehat{g})\\
        &= \frac{\eps}{n^2}\sum_{j=1}^n\Bigl\{\psi\biggl(\frac{{\widehat{f}}(X_1')+{\widehat{g}}(Y_j)-c(X'_1,Y_j)}{\eps} \biggr)- \psi\biggl(\frac{{\widehat{f}}(X_1)+{\widehat{g}}(Y_j)-c(X_1,Y_j)}{\eps} \biggr) \Bigr\}
        \\
        &\quad\, +\frac{{\widehat{f}}(X_1)-{\widehat{f}}(X_1')}{n}  .
    \end{align*}
   Since $\psi$ is convex,  it follows that
   \begin{align*}
      \widehat{\Phi}_{\phi,\eps}& ({\widehat{f}}, \widehat{g}) -\widehat{\Phi}_{\phi,\eps}^{(1)} ({\widehat{f}}^{(1)}, \widehat{g}^{(1)})
    \leq  \frac{{\widehat{f}}(X_1)-{\widehat{f}}(X_1')}{n}  \\
   &\qquad+ \frac{1}{n^2}
   \sum_{j=1}^n\Bigl\{\widehat{f}(X_1')-{\widehat{f}}(X_1)- c(X_1',Y_j)+c(X_1,Y_j)  \Bigr\}\psi'\biggl(\frac{{\widehat{f}}(X_1')+{\widehat{g}}(Y_j)-c(X'_1,Y_j)}{\eps} \biggr).
   \end{align*}
Choosing the pointwise extension of $\widehat f$ given by \cref{pr:ROTprelims}(vi), the empirical first-order condition holds at $X'_1$. Using this optimality condition and the boundedness of $c$ (see \cref{pr:ROTprelims}), it follows that
  \begin{align*}
      \widehat{\Phi}_{\phi,\eps} ({\widehat{f}}, \widehat{g}) -&\widehat{\Phi}_{\phi,\eps}^{(1)} ({\widehat{f}}^{(1)}, \widehat{g}^{(1)})
    \\
    &\leq  \frac{1}{n^2}\sum_{j=1}^n\psi'\biggl(\frac{{\widehat{f}}(X_1')+{\widehat{g}}(Y_j)-c(X'_1,Y_j)}{\eps} \biggr)  (c(X_1,Y_j)-c(X_1',Y_j) )\leq  \frac{2\|c\|_\infty}{n}. 
   \end{align*}
By symmetry, this implies $\widehat{\Phi}_{\phi,\eps} ({\widehat{f}}, \widehat{g}) -\widehat{\Phi}_{\phi,\eps}^{(1)} ({\widehat{f}}^{(1)}, \widehat{g}^{(1)})\geq -\frac{2\|c\|_\infty}{n}$. We conclude that  $A\leq \frac{4\|c\|_\infty^2}{n^2} $, and analogously, $B\leq \frac{4\|c\|_\infty^2}{n^2}$. The claim follows.
\end{proof}

\begin{proof}[Proof of \Cref{lemma:Bound-lipschitz}]
Fix $(x_0,y_0)\in \Omega\times\Omega'$.
By \Cref{pr:ROTprelims}, the function 
$$  x\mapsto {f_\eps(x)+g_\eps(y_0)-c(x,y_0)}  $$
is Lipschitz with constant  $2L$. Hence, 
$${f_\eps(x)+g_\eps(y_0)-c(x,y_0)} \geq {f_\eps(x_0)+g_\eps(y_0)-c(x_0,y_0)}- 2L\|x-x_0\|,$$ 
and by the monotonicity of $\psi'$ and~\eqref{eq:dual-primal}, 
$$ \rho_\eps(x,y_0)  \geq \psi'\biggl(\frac{f_\eps(x_0)+g_\eps(y_0)-c(x_0,y_0)- 2L\|x-x_0\|}{\eps} \biggr). $$
Integrating with respect to $P$, it follows that 
$$ 1  \geq \int \psi'\biggl(\frac{f_\eps(x_0)+g_\eps(y_0)-c(x_0,y_0)}{\eps}- \frac{2L}{\eps}\|x-x_0\| \biggr) dP(x)  . $$
Set $A=\frac{f_\eps(x_0)+g_\eps(y_0)-c(x_0,y_0)}{\eps}$.
Then, for every $r>0$, 
 $$ 1\geq \int_{\overline{\mathbb{B}}_{r}(x_0)}  \psi'\biggl(A- \frac{2L}{\eps}\|x-x_0\| \biggr) dP(x)\geq \psi'\biggl(A- \frac{2L}{\eps}r \biggr) P(\overline{\mathbb{B}}_{r}(x_0)), $$
which implies  
 $ \frac{1}{P(\overline{\mathbb{B}}_{r}(x_0))}\geq \psi'\big(A- \frac{2L}{\eps}r \big)   $. Since $\phi'$ is injective in $[1, \infty)$ by \cref{assumption:divergence}, it follows that 
 $$ A \leq  \phi'\biggl( \frac{1}{P(\overline{\mathbb{B}}_{r}(x_0))} \biggr)+ \frac{2L}{\eps} r,$$
 so that $ \rho_\eps(x_0,y_0)\leq \inf_{r>0}\psi'\big( \phi'\big( \frac{1}{P(\overline{\mathbb{B}}_{r}(x_0))} \big)+ \frac{2L}{\eps} r \big).$ We conclude by the bound 
 \begin{align*}
     \int \bigl(\rho_\eps(x,y)\bigr)^2 d(P\otimes Q)(x,y)&\leq \int \rho_\eps(x,y)  \inf_{r>0}\psi'\biggl( \phi'\biggl( \frac{1}{P(\overline{\mathbb{B}}_{r}(x))} \biggr)+ \frac{2L}{\eps} r \biggr) d(P\otimes Q)(x,y)\\
     &= \int \inf_{r>0}\psi'\biggl( \phi'\biggl( \frac{1}{P(\overline{\mathbb{B}}_{r}(x))} \biggr)+ \frac{2L}{\eps} r \biggr) dP(x),
 \end{align*}
where the last equality follows from the fact that $\pi_\eps$ has first marginal $P$. 
\end{proof}

\begin{proof}[Proof of \cref{Theorem:CLT}]
We use
the Efron--Stein inequality \cite[Section~3.1]{Boucheron.el.al.2013.book} with the estimator 
$$ \widehat{W}= \widehat{\rm OT}_{\phi,\eps} - \int h_\eps(x,y) d(\widehat{P} \otimes \widehat{Q})(x,y)$$
to get
$$ {\rm Var} \bigl[\widehat{W}\bigr] \leq n \biggl(  \EE\bigl[ \bigl(\widehat{W}-\widehat{W}^{(1)}\bigr)_+^2\bigr]+ \EE\bigl[ \bigl(\widehat{W}-\widehat{W}_{(1)}\bigr)_+^2\bigr] \biggr), $$
where $\widehat{W}^{(1)}$ (resp.~$\widehat{W}_{(1)}$) denotes the estimator based on the sample $(X'_1, X_2,\dots, X_n)$ and $(Y_1, Y_2,\dots, Y_n)$ (resp.~$(X_1, X_2,\dots, X_n)$ and $(Y'_1, Y_2,\dots, Y_n)$) where $X'_1$ is an independent copy of $X_1$ and $Y'_1$ is an independent copy of $Y_1$. It is easy to check that
$$ \widehat{W}-\widehat{W}^{(1)} \leq    \frac{1}{n} \int \widehat{h}(X_1,y)- {h}_\eps(X_1,y)-\bigl(\widehat{h}(X_1',y)- {h}_\eps(X_1',y)\bigr) d\widehat{Q}(y), $$
where $\widehat{h}$ denotes the empirical version of $h_\eps$. Hence,
$$ \EE\bigl[ \bigl(\widehat{W}-\widehat{W}^{(1)}\bigr)_+^2\bigr] \leq    \frac{1}{n^2} \EE\biggl[\Bigl(\widehat{h}(X_1,Y_1)- {h}_\eps(X_1,Y_1)-\widehat{h}(X_1',Y_1)+ {h}_\eps(X_1',Y_1)\Bigr)^2 \biggr]. $$
Using \cref{lemma:stability}  and the fact that with $\mathbb{P}$-probability one, $ \widehat{P}$ and $\widehat{Q}$ converge in distribution to $P$ and $Q$, we derive $ n {\rm Var}[\widehat{W}] \to 0 $.  By \cref{pr:general,lemma:stability}  we also have 
$\EE[\widehat{\rm OT}_{\phi,\eps}]- {\rm OT}_{\phi,\eps}=o(n^{-\frac{1}{2}})$.
Hence,  $  \sqrt{n}(\widehat{\rm OT}_{\phi,\eps}-\EE[\widehat{\rm OT}_{\phi,\eps}]) $ and $ \sqrt{n}\int h_\eps(x,y) d(\widehat{P} \otimes \widehat{Q}-{P} \otimes {Q})(x,y)$ converge to the same Gaussian variable, and the claim follows. 
\end{proof}

\begin{proof}[Proof of \cref{lemma:Bound-energy-Lp}]
    For $p\geq 2$, we have $\psi'_p(t) = (t)_+^{\frac{1}{p-1}}$. We use \cref{lemma:Bound-lipschitz} and 
    $$ \bigl(a+b\bigr)_+^{\alpha} \leq  \bigl(a\bigr)_+^{\alpha} +  \bigl(b\bigr)_+^{\alpha}, \quad \alpha \in (0,1]  $$
    to see that for $r>0$,
    \begin{align*}
         \int  \psi'_p\biggl( \phi'_p\biggl( \frac{1}{P(\overline{\mathbb{B}}_{r}(x))} \biggr)+  \frac{2L}{\eps} r \biggr) dP(x) &\leq  \int  \psi'_p\biggl( \phi'_p\biggl( \frac{1}{P(\overline{\mathbb{B}}_{r}(x))} \biggr)\biggr)  dP(x) + \psi'_p\biggl(\frac{2L}{\eps} r \biggr)\\
         &\leq    \int   \frac{1}{P(\overline{\mathbb{B}}_{r}(x))}   dP(x) + \psi'_p\biggl(\frac{2L}{\eps} r \biggr)\\
         &\leq    \mathcal{N}\bigl(\Omega,\frac{r}{4}\bigr)+  \bigl( \frac{2L}{\eps} r \bigr)^{\frac{1}{p-1}},
    \end{align*}
where the last inequality is shown in \cite[Proposition~18]{Stromme.24}. The other two cases follow by the same argument, except that for (ii) we use 
$$ \bigl(a+b\bigr)_+^{\alpha} \leq  2^{\alpha-1}\bigl(a\bigr)_+^{\alpha} +   2^{\alpha-1}\bigl(b\bigr)_+^{\alpha}, \quad  \alpha >1$$
and for (iii) we use  
\begin{equation*}
\exp\biggl( \log\biggl( \frac{1}{P(\overline{\mathbb{B}}_{r}(x))} \biggr)+  \frac{2L}{\eps} r \biggr)=   \frac{\exp\bigl( \frac{2L}{\eps} r  \bigr) }{P(\overline{\mathbb{B}}_{r}(x))}. \mbox{\qedhere}
\end{equation*}
\end{proof}

\begin{proof}[Proof of \cref{example:sharp}] Since 
$1=\int  \rho_\eps d(P\otimes Q)$ does not vary with $\eps$, it is enough to show the claim for $\int  \rho_\eps^2 d(P\otimes Q)$ instead of the variance. 
   As pointed out by \cite{gonzalezsanz.2025.sparseregularizedoptimaltransport}, in this setting the ROT potentials are explicit constant functions: $f_\eps\equiv c_\eps$ and $g_\eps\equiv0$, where $c_\eps\in\R$ solves
   $$ 1=\int_{[0,1]^d} \psi'_p\biggl( \frac{c_\eps-d_{\mathbb{T}^d}(x,0)}{\eps}\biggr) dx.$$
   For EOT (i.e., $p=1$), we have
   $$ e^{-\frac{c_\eps}{\eps}}=e^{-1} \int_{[0,1]^d}  e^{\frac{-d_{\mathbb{T}^d}(x,0)}{\eps}} dx \approx \eps^d,$$ which implies
   \begin{align*}
       \int_{[0,1]^d} \int_{[0,1]^d} \exp\biggl( \frac{2c_\eps-2d_{\mathbb{T}^d}(x,y)}{\eps}\biggr) dx dy &\gtrsim  \eps^{-2d} \int_{[0,1]^d} \exp\biggl(-\frac{2d_{\mathbb{T}^d}(x,0)}{\eps}\biggr) dx \approx \eps^{-d}. 
   \end{align*}
For $p>1$, we have
$$ 1=\int_{[0,1]^d} \biggl( \frac{c_\eps-d_{\mathbb{T}^d}(x,0)}{\eps}\biggr)_+^{q-1} dx,$$
where $q=p/(p-1)$ is the conjugate exponent of $p$. Since $c_\eps\to 0 $ as $\eps\to 0$ (for otherwise the previous equality cannot hold), it follows that  
$$ 1= \int_{\mathbb{B}_{\frac{1}{4}}(0)} \biggl( \frac{c_\eps-d_{\mathbb{T}^d}(x,0)}{\eps}\biggr)_+^{q-1} dx = \int_{\mathbb{B}_{\frac{1}{4}}(0)} \biggl( \frac{c_\eps-\|x\|}{\eps}\biggr)_+^{q-1} dx= \int  \biggl( \frac{c_\eps-\|x\|}{\eps}\biggr)_+^{q-1} dx$$
for $\eps$ small enough. Solving the last integral in spherical coordinates, we get 
$ c_\eps \approx  \eps^{\frac{q-1}{q+d-1}}$ and hence 
$$ \int_{[0,1]^d}\int_{[0,1]^d} \biggl( \frac{c_\eps-d_{\mathbb{T}^d}(x,y)}{\eps}\biggr)_+^{2(q-1)} dx dy \approx \eps^{-\frac{d (q-1)}{ (q-1)+d}}=\eps^{-\frac{d }{  1+(p-1)d}}, $$
concluding the proof. 
\end{proof}

\begin{proof}[Proof of \cref{lemma:Minkowski-quantization}]
Fix $D_P> {\rm dim}_M(\Omega)$, so that there exist $C_\Omega,\beta_0>0$ such that for every $\beta\leq \beta_0$,
\[
\mathcal{N}(\beta)=\mathcal{N}(\Omega, \beta) \leq C_\Omega \beta^{-D_P}.
\]
Fix $\beta_n= C_\Omega^{1/D_P} n^{-1/D_P}$ and $n_0$ such that $\beta_n\leq \beta_0$ for all $n\geq n_0$. Then $\mathcal{N}(\beta_n)\leq n$, and hence there exists a sequence
$S_1, \dots, S_{\mathcal{N}(\beta_n)}$ of disjoint sets with diameter at most $2\beta_n$ whose union contains $\Omega$. We choose $x_i \in S_i$ for all $i\in \{ 1, \dots, \mathcal{N}(\beta_n)\} $. Next, we define the probability measure $P_n=\sum_{i=1}^{\mathcal{N}(\beta_n)} \delta_{x_i}  P(S_i)$ 
and the transport map
\[
x\mapsto  T(x)=\sum_{i=1}^{\mathcal{N}(\beta_n)} x_i {\bf 1}_{S_i}.
\]
Since $T$ pushes $P$ forward to $P_n$, we have
\[
\mathcal{W}_2^2(P_n,  P)
\leq
\sum_{i=1}^{\mathcal{N}(\beta_n)}
\int_{S_i} \|x_i- x\|^2 dP(x)
\leq
\sum_{i=1}^{\mathcal{N}(\beta_n)}
4\beta_{n}^2 P(S_i)
= 4\beta_{n}^2,
\]
so that
\[
\mathcal{W}_2(P_n,  P) \leq 2 C_\Omega^{1/D_P} n^{-1/D_P}.
\]
This concludes the proof of~(i) and~(ii). 
For (iii) and (iv), we slightly modify the preceding choice of scale. Fix
\[
\widetilde\beta_n := 2 C_\Omega^{1/D_P} n^{-1/D_P}
\]
and increase $n_0$ if necessary so that $\widetilde\beta_n/2\leq \beta_0$ for all $n\geq n_0$. We use a maximal $\widetilde\beta_n$-packing of $\Omega$, i.e. a finite set
$\{x_1,\dots,x_N\}\subset\Omega$ such that
\[
\|x_i-x_j\|>\widetilde\beta_n,\qquad i\neq j,
\]
and such that no further point of $\Omega$ can be added while preserving this property.
Since any $\widetilde\beta_n$-packing contains at most one point in each ball of a
$\widetilde\beta_n/2$-cover, we have
\[
N\leq \mathcal N(\Omega,\widetilde\beta_n/2)
\leq C_\Omega(\widetilde\beta_n/2)^{-D_P}
= n.
\]
Moreover, by maximality, $\{x_1,\dots,x_N\}$ is a $\widetilde\beta_n$-cover of $\Omega$.
Because the balls $\mathbb B_{\widetilde\beta_n/2}(x_i)$ are pairwise disjoint, we may choose a measurable partition
$S_1,\dots,S_N$ of $\Omega$ such that
\[
\mathbb B_{\widetilde\beta_n/2}(x_i)\cap\Omega
\subset S_i
\subset
\mathbb B_{\widetilde\beta_n}(x_i)\cap\Omega,
\qquad i=1,\dots,N.
\]
Define $P_n=\sum_{i=1}^N \delta_{x_i}P(S_i).$
Since $N\leq n$, the measure $P_n$ has support of cardinality at most $n$. Moreover, the transport map
$T(x)=\sum_{i=1}^N x_i{\bf 1}_{S_i}(x)$ pushes $P$ forward to $P_n$, and the inclusion
$S_i\subset \mathbb B_{\widetilde\beta_n}(x_i)$ gives
\[
\mathcal W_2(P_n,P)
\leq \widetilde\beta_n
=
2 C_\Omega^{1/D_P} n^{-1/D_P}.
\]
Finally, using \eqref{eq:Bound-uniform-quant-lemma}, for every $i,j$ we obtain
\[
P(S_i)
\leq P\bigl(\mathbb B_{\widetilde\beta_n}(x_i)\bigr)
\leq
C\cdot P\bigl(\mathbb B_{\widetilde\beta_n/2}(x_j)\bigr)
\leq C\cdot P(S_j).
\]
Thus
\[
\frac{\min_i P(S_i)}{\max_i P(S_i)}\geq \frac1C,
\]
uniformly in $n$. Hence the approximating measures $P_n$ satisfy the required uniform mass comparability, and (iii) and (iv) follow.
\end{proof}

\begin{proof}[Proof of~\Cref{prop:improved_qot_quadratic_cost}]
The particularity of the quadratic regularizer is that,
with the normalization $\phi_2(t)=(t^2-1)/2$ used for the Tsallis family,
\[
D_{\phi_2}(\pi_\eps\mid P\otimes Q)
=\frac12\int (\rho_\eps^2-1)\,d(P\otimes Q)
=\frac12\,{\rm Var}_{P\otimes Q}[\rho_\eps].
\]
Hence, as $\pi_\eps\in\Pi(P,Q)$,
\[
{\rm OT}_{\phi_2,\eps}-{\rm OT}
=
\int c\,d\pi_\eps-{\rm OT}
+\frac{\eps}{2}\,{\rm Var}_{P\otimes Q}[\rho_\eps]
\geq
\frac{\eps}{2}\,{\rm Var}_{P\otimes Q}[\rho_\eps].
\]
Recalling that $c(x,y)=\|x-y\|^2$, \cite[Corollary~3.14]{EcksteinNutz.22} yields the approximation error
\[
0\leq {\rm OT}_{\phi_2,\eps}-{\rm OT}
\lesssim \eps^{\frac{2}{d+2}},
\qquad \eps\in(0,1].
\]
Combining the last two displays gives
\[
{\rm Var}_{P\otimes Q}[\rho_\eps]
\lesssim
\eps^{\frac{2}{d+2}-1}
=
\eps^{-\frac{d}{d+2}}.
\]

Since $\Omega,\Omega'$ are bounded,  $5\|c\|_\infty+\eps\phi_2'(1)\leq C, $
for some $C>0$. 
Taking square roots in the second estimate of~\Cref{pr:general} therefore yields
\[
\bigl|\EE\bigl[\widehat{\rm OT}_{\phi_2,\eps}\bigr]
-{\rm OT}_{\phi_2,\eps}\bigr|
\lesssim
\frac{{\rm Var}_{P\otimes Q}[\rho_\eps]^{1/2}}{\sqrt n}
\lesssim
\frac{\eps^{-\frac{d}{2d+4}}}{\sqrt n},
\]
and then adding the approximation error once more gives
\[
\bigl|\EE\bigl[\widehat{\rm OT}_{\phi_2,\eps}\bigr]-{\rm OT}\bigr|
\leq
\bigl|\EE\bigl[\widehat{\rm OT}_{\phi_2,\eps}\bigr]
-{\rm OT}_{\phi_2,\eps}\bigr|
+{\rm OT}_{\phi_2,\eps}-{\rm OT}
\lesssim
\frac{\eps^{-\frac{d}{2d+4}}}{\sqrt n}
+\eps^{\frac{2}{d+2}}.
\]
Finally, balancing the two powers of $\eps$ gives
$\eps\asymp n^{-(d+2)/(d+4)}$, and now substitution yields the second claim.
\end{proof}

 \section{Proofs of \Cref{Section-compl-fast}}\label{sect-proof-pot}

Let \cref{assumption:lipschitz-cost,assumption:divergence-strong,assumption-P-strong} hold. We start by studying the (random) convex function $$\Gamma(t) = -\widehat{\Phi}_{\phi,\eps}\bigl( (1-t)({\widehat{f}},{\widehat{g}} )  + t\, ({f}_\eps,g_\eps)\bigr), \quad t\in [0,1]. $$ 

\begin{lemma}\label{lemma:taylor}
The function $\Gamma$ is convex, $\mathcal{C}^{1,1}$ and has Taylor expansion, for $t\in [0,1]$, 
$$  \Gamma(t) =\Gamma(0)+  \int_{0}^t \int_{0}^s   \frac{1}{\eps} \int  \Bigl(\bigl( \widehat{f} - f_\eps \bigr) \oplus \bigl( \widehat{g} - g_\eps \bigr) \Bigr)^2\psi''(\gamma_r)  d(\widehat{P}\otimes \widehat{Q}) dr ds, $$
where 
$$ \gamma_t =\frac{\bigl((1-t){\widehat{f}}+ t f_\eps\bigr)\oplus \bigl((1-t){\widehat{g}}+ t g_\eps\bigr) - c}{\eps} .$$  
\end{lemma}
\begin{proof} 
Put
\begin{multline*}
    h_{i,j}(t)= -(1-t)\bigl({\widehat{f}}(X_i)+ {\widehat{g}}(Y_j)\bigr)  - t \bigl({f}_\eps (X_i )+ {g}_\eps (Y_j )\bigr)  \\+\eps \cdot \psi\biggl(\frac{ (1-t)\bigl({\widehat{f}}(X_i)+ {\widehat{g}}(Y_j)\bigr)  + t\, \bigl({f}_\eps (X_i )+ {g}_\eps (Y_j )\bigr)-c(X_i,Y_j)}{\eps}\biggr),
\end{multline*}
which is a $\mathcal{C}^{1,1}$ function by~\Cref{assumption:divergence-strong}.
It has first derivative 
$$ h_{i,j}'(t) =\Bigl( {f}_\eps (X_i )+ {g}_\eps (Y_j )- \bigl({\widehat{f}}(X_i)+ {\widehat{g}}(Y_j)\bigr) \Bigr) \bigl( \psi'\bigl(\gamma_t(X_i,Y_j)\bigr) - 1 \bigr).$$
The second derivative exists almost everywhere
by Rademacher's theorem. Indeed,
$$  h_{i,j}''(t) =\frac{1}{\eps}\psi'' \bigl(\gamma_t(X_i,Y_j)\bigr) \Bigl( {f}_\eps (X_i )+ {g}_\eps (Y_j )- \bigl({\widehat{f}}(X_i)+ {\widehat{g}}(Y_j)\bigr) \Bigr)^2$$
for Lebesgue-a.e.~$t\in \R$.  Applying the fundamental theorem of calculus twice yields 
\begin{align}
 h_{i,j}(t) =h_{i,j}(0)+  \int_{0}^t h_{i,j}'(s) ds =h_{i,j}(0)+ h_{i,j}'(0) t   +\int_{0}^t  \int_{0}^s h_{i,j}''(r)\,  dr\,ds . 
\end{align}
Hence, almost surely,
$$ \frac{1}{n^2} \sum_{i,j} h_{i,j}(t) = \frac{1}{n^2} \sum_{i,j} h_{i,j}(0)+ \frac{t}{n^2} \sum_{i,j}h_{i,j}'(0)   +\int_{0}^t  \int_{0}^s \frac{1}{n^2} \sum_{i,j} h_{i,j}''(r)  \,dr\,ds.$$
Noting that $\frac{1}{n^2} \sum_{i,j}h_{i,j}'(0)= 0$,   and  $\frac{1}{n^2} \sum_{i,j} h_{i,j}(0)= \Gamma(0)$ the result follows.  
\end{proof}
In view of \Cref{lemma:taylor}, we want to find a lower bound for $\psi''(\gamma_r)$. 
This will be achieved by showing that the empirical bilinear form 
$$   \widehat{B}_r( u\oplus v,\tilde u\oplus \tilde v ) := \int \psi''(\gamma_r) \bigl(u\oplus v\bigr)  \bigl(\tilde u\oplus \tilde v\bigr)   d(\widehat{P}\otimes \widehat{Q}) $$
is coercive in the finite-dimensional random Hilbert space
$$\widehat{L}_\oplus^2:={\rm sp}\bigl\{ u\oplus v: (u,v)\in L^2(\widehat{P})\times L^2(\widehat{Q})\bigr\}
\subset L^2(\widehat{P}\otimes \widehat{Q}) $$
endowed with the induced inner product. First, we will show in \cref{se:component-wise-coercivity} the coercivity with respect to the first component, i.e., 
$$   \widehat{B}_r( u\oplus v,u\oplus v ) \geq \widehat{\alpha} {\rm Var}_{\widehat{P}}[u],  $$
for a sequence of random variables $\widehat{\alpha}$ (indexed by~$n$). Then, we will use this component-wise coercivity in \cref{se:coercivity-joint} to show that 
$$   \widehat{B}_r( u\oplus v,u\oplus v ) \geq \widehat{\beta} \|u\oplus v\|_{L^2(\widehat{P}\otimes \widehat{Q})}^2,  $$
for a different sequence of random variables $\widehat{\beta}$.
Finally, in \cref{se:deterministic-bounds}, we will obtain a suitable control for~$\widehat{\beta}$. In the remainder of the paper we abbreviate 
$$t_\eps:= \min\left(\frac{\eps \delta_\psi}{8\bigl(5\|c\|_\infty+\eps\phi'(1)\bigr) }, 1\right).$$
  
  \subsection{Component-wise coercivity}\label{se:component-wise-coercivity}

Define the event
     $$\mathcal{E}_n= \biggl\{  \text{there exists $\hat{T}$ s.t.~$\hat{T}_\# P =\widehat{P}$ and $  \|\hat{T}-{\rm I}\|_{L^\infty(P)} \leq \frac{\delta_\psi\eps}{32 L}$}\biggr\}\subset \boldsymbol{\Omega} . $$ 
The main result of this subsection is the following lower bound.     
\begin{proposition}\label{proposition:component-coercivity}
    For every $0\leq r\leq t_\eps$ and all $u\oplus v \in \widehat L_\oplus^2$,  
    $$  \widehat{B}_r( u\oplus v,u\oplus v )\geq  \widehat{\alpha} \, {\rm Var}_{\widehat{P}}[u] {\bf 1}_{\mathcal{E}_n}, $$
    where 
    $$ \widehat{\alpha}:=\frac{ \lambda_P^2\psi''\bigl(t_\psi -\delta_\psi  \bigr)^2 \min_{j}  \widehat{Q}\Bigl(\mathbb{B}_{\frac{\delta_\psi \eps}{8L}}(Y_j)\Bigr)}{2\gamma(n,\eps) \Lambda_P^2\Bigl(\Bigl\lceil\frac{16L{\rm diam}(\Omega)}{\delta_\psi\eps} \Bigr\rceil \Bigr)^{d+2}} . $$
\end{proposition}
The rest of this subsection is devoted to the proof of \Cref{proposition:component-coercivity}. 
We start with an auxiliary lemma; it reduces the problem to bounding below  the first  nontrivial eigenvalue of the Laplacian of the random graph with weights $w_{i,\ell}= {\bf 1}_{\|X_i-X_\ell\| \leq \frac{\delta_\psi \eps}{8L}}$.

\begin{lemma}\label{Lemma:From two-to-one-variance}
Let
$$  \widehat{\Theta}_\eps(u,\tilde u) = \frac{1}{n^2}\sum_{\|X_i-X_\ell\| \leq \frac{\delta_\psi \eps}{8L}}     \bigl(u(X_i)- u(X_\ell) \bigr)\bigl(\tilde u (X_i)- \tilde u (X_\ell) \bigr) . $$
Then for every $0\leq r\leq t_\eps$ and $u\oplus v \in \widehat L_\oplus^2$,  
$$
         \widehat{B}_r( u\oplus v,u\oplus v )
     \geq  {\frac12}\Biggl(  \frac{ \min_{j}  \widehat{Q}\Bigl(\mathbb{B}_{\frac{\delta_\psi \eps}{8L}}(Y_j)\Bigr)}{\max_j\frac{1}{n}\sum_{k} \pi_{k,j}^{(1)}} \Biggr)\psi''(t_\psi -\delta_\psi)^2 \widehat{\Theta}_\eps(u, u) . 
$$

\end{lemma}

\begin{proof}
\emph{Step~1: Reduction to pairwise $X$-differences.}
Fix $0\leq r\leq t_\eps$ and $u\oplus v \in \widehat L_\oplus^2$. \Cref{assumption:divergence-strong} 
states that $\psi''$ is monotone, so \Cref{pr:ROTprelims} implies that
\[
\psi''(\gamma_r)\geq \psi''\biggl(\frac{\widehat{f}\oplus \widehat{g}-c }{\eps}
-\frac{\delta_\psi}{4}\biggr).
\]
Integrating against $\widehat{P}\otimes \widehat{Q}$, we obtain
\[
\int \psi''(\gamma_r) \bigl(u\oplus v\bigr)^2  d(\widehat{P}\otimes \widehat{Q})
\geq
\int \psi''\biggl(\frac{\widehat{f}\oplus \widehat{g}-c}{\eps}
-\frac{\delta_\psi}{4} \biggr)
\bigl(u\oplus v\bigr)^2  d(\widehat{P}\otimes \widehat{Q}).
\]
Define 
\[
a_{i,j}:=\psi''\Bigl(\frac{\widehat f(X_i)+\widehat g(Y_j)-c(X_i,Y_j)}{\eps}
-\frac{\delta_\psi}{4}\Bigr)\leq \pi_{i,j}^{(1)}.
\]
Hence, with the convention that
$\frac{\sum_{k} a_{k,j} u(X_k)}{\sum_{k} a_{k,j}}=0$ when $\sum_k a_{k,j}=0$,
it follows that 
\begin{align*}
    \int \psi''(\gamma_r) \bigl(u\oplus v\bigr)^2  d(\widehat{P}\otimes \widehat{Q})
    &\geq \frac{1}{n^2} \sum_{j} \min_{s\in \R}\sum_i
    a_{i,j}\bigl(u(X_i)- s\bigr)^2  \\
    &= \frac{1}{ n^2} \sum_{j,i}
    a_{i,j} \biggl(u(X_i)-
    \frac{\sum_{k} a_{k,j} u(X_k)}{\sum_{k} a_{k,j} } \biggr)^2  \\
    &=
    \frac{1}{2 n^2}
    \sum_{\substack{j:\,\sum_k a_{k,j}>0}}\sum_{i,\ell}
    \frac{a_{i,j} a_{\ell, j}}{\sum_k a_{k, j}}
    \bigl(u(X_i) - u(X_\ell)\bigr)^2.
\end{align*}
As a consequence, since $a_{k,j}\leq \pi_{k,j}^{(1)}$,
\begin{align}\label{Bipartite}
    \int \psi''(\gamma_r) \bigl(u\oplus v\bigr)^2  d(\widehat{P}\otimes \widehat{Q})
    &\geq
    \frac{1}{2 n^2}\sum_{j,i,\ell}
    \frac{a_{i,j} a_{\ell,j}}{\sum_{k} \pi_{k,j}^{(1)}}
    \bigl(u(X_i)- u(X_\ell) \bigr)^2 .
\end{align}

\emph{Step~2: Localization to the weighted graph.}
For each $i\in \{1, \dots, n\}$, fix
\[
j(i)\in \argmax_{j}\widehat{f}(X_i)+ \widehat{g}(Y_j)-c(X_i,Y_j)
\]
and define 
\[
m_i= \max_{j}\widehat{f}(X_i)+ \widehat{g}(Y_j)-c(X_i,Y_j).
\]
We use the fact that $\psi'$ is monotone to get 
\[
1= \frac{1}{n} \sum_{j=1}^n
\psi'\biggl(\frac{\widehat{f}(X_i)+ \widehat{g}(Y_j)-c(X_i,Y_j) }{\eps}\biggr)
\leq \psi'\biggl(\frac{m_i }{\eps}\biggr),
\]
so that $\frac{m_i }{\eps} \geq t_\psi$. Next, since
$\widehat{f}\oplus \widehat{g}-c$ is $2L$-Lipschitz in each variable,
we obtain
\begin{align*}
a_{\ell, j}
&\geqslant 
\psi'' \biggl(\frac{\widehat f(X_i) + \widehat g(Y_j) - c(X_i, Y_j)
- 2L\|X_i - X_\ell\|}{\eps}- \frac{\delta_\psi}{4}  \biggr) \\
&\geqslant
\psi''\biggl(\frac{m_i - 2L\|X_i - X_\ell\|
- 2L\|Y_j - Y_{j(i)}\|}{\eps} - \frac{\delta_\psi}{4}  \biggr).
\end{align*}
Similarly, we obtain
\[
a_{i, j} \geqslant
\psi''\biggl(\frac{m_i - 2L\|X_i - X_\ell\|
- 2L\|Y_j - Y_{j(i)}\|}{\eps}- \frac{\delta_\psi}{4}  \biggr).
\]
Hence, using also $\frac{m_i}{\eps} \geq t_\psi$,
\begin{align*}
a_{i,j}  a_{\ell,j}
&\geq
\Bigl(\psi''\biggl(\frac{m_i-2L\bigl(\|X_i-X_\ell\|
+ \|Y_{j(i)}-Y_j\|\bigr)}{\eps} - \frac{\delta_\psi}{4} \biggr) \Bigr)^2  \\
&\geq
\psi''\bigl(t_\psi -\delta_\psi \bigr)^2
{\bf 1}_{ \|X_i-X_\ell\| \leq \frac{\delta_\psi\eps}{8L}}
{\bf 1}_{ \|Y_{j(i)}-Y_j\| \leq \frac{\delta_\psi\eps}{8L}}.
\end{align*}
Thus~\eqref{Bipartite} yields
\begin{align*}
\int &\psi''(\gamma_r) \bigl(u\oplus v\bigr)^2  d(\widehat{P}\otimes \widehat{Q})\\
&\geq
{\frac{1}{2}}\frac{\psi''(t_\psi -\delta_\psi )^2}{n^2}
\sum_{j,i,\ell}
\bigl(u(X_i)- u(X_\ell) \bigr)^2
\frac{
{\bf 1}_{ \|X_i-X_\ell\| \leq \frac{\delta_\psi \eps}{8L}}
{\bf 1}_{ \|Y_{j(i)}-Y_j\| \leq \frac{\delta_\psi \eps}{8L}}
}{\sum_{k} \pi_{k,j}^{(1)}}\\
&=
{\frac{1}{2}}\frac{\psi''(t_\psi -\delta_\psi)^2}{n^2}
\sum_{i,\ell}
\bigl(u(X_i)- u(X_\ell) \bigr)^2
{\bf 1}_{ \|X_i-X_\ell\| \leq \frac{\delta_\psi \eps}{8L}}
\sum_{j}
\frac{
{\bf 1}_{ \|Y_{j(i)}-Y_j\| \leq \frac{\delta_\psi \eps}{8L}}
}{\sum_{k} \pi_{k,j}^{(1)}}\\
&\geq
{\frac{1}{2}}
\Biggl(
\frac{
\min_{j} \widehat{Q}\Bigl(\mathbb{B}_{\frac{\delta_\psi \eps}{8L}}(Y_j)\Bigr)
}{
\max_j\frac{1}{n}\sum_{k} \pi_{k,j}^{(1)}
}
\Biggr)
\frac{\psi''(t_\psi -\delta_\psi)^2}{n^2}
\sum_{i,\ell}
\bigl(u(X_i)- u(X_\ell) \bigr)^2
{\bf 1}_{ \|X_i-X_\ell\| \leq \frac{\delta_\psi \eps}{8L}}\\
&=
{\frac{1}{2}}
\Biggl(
\frac{
\min_{j} \widehat{Q}\Bigl(\mathbb{B}_{\frac{\delta_\psi \eps}{8L}}(Y_j)\Bigr)
}{
\max_j\frac{1}{n}\sum_{k} \pi_{k,j}^{(1)}
}
\Biggr)
\psi''(t_\psi -\delta_\psi)^2
\widehat{\Theta}_\eps(u,u).
\end{align*}
This proves the claim.
\end{proof}

Next, we compare $\widehat{\Theta}_\eps(u,u)$ with
${\rm Var}_{\widehat{P}}[u]$. Our proof adapts the method of
\cite{Trillos-manifold-laplace}, originally developed to transfer coercivity and
spectral information between random geometric graph Laplacians and their
continuum counterparts.

\begin{proposition}\label{pr:comparison-Theta-Var}For all $u\in L^2(\widehat{P})$, we have
    $$ \widehat{\Theta}_\eps(u,u)\geq  {\frac{\lambda_P^2}{\Lambda_P^2\Bigl(\Bigl\lceil\frac{16 L{\rm diam}(\Omega)}{\delta_\psi\eps} \Bigr\rceil \Bigr)^{d+2}}} {\rm Var}_{\widehat{P}}[u]{\bf 1}_{\mathcal{E}_n}.$$
\end{proposition}

The proof will be stated at the end of this subsection, after some auxiliary results. For a measurable map $T$ such that $T_\# P =\widehat{P}$, define the operators 
\begin{align*}
    \Pi_T: L^2(\widehat{P})& \to L^2(P) \\
   f &\mapsto  \sum_{i}  f(X_i) {\bf 1}_{T^{-1}(X_i)}=  f\circ T 
\end{align*}
and 
\begin{align*}
    \Pi_T^*: L^2(P)& \to L^2(\widehat{P}) \\
   g &\mapsto  n \cdot \biggl(\int_{T^{-1}(X_1)} g(x) dP(x), \dots, \int_{T^{-1}(X_n)} g(x) dP(x) \biggr).
\end{align*}
We then have the following result. 
\begin{lemma}
    \label{lemma:Isometry} The operator $\Pi_T^*$ is the adjoint of $\Pi_T$, i.e., for every $g\in L^2(P)$ and $f\in L^2(\widehat{P})$, 
    $$  \int (\Pi_T f)(x) g(x) dP(x)= \int (\Pi_T^* g)(x) f(x) d\widehat{P}(x) .$$
Moreover, $\Pi_T$ is an isometry. 
\end{lemma}
\begin{proof}
Note
$$ \int (\Pi_T f)(x) g(x) dP(x)= \sum_{i} f(X_i)\int_{T^{-1}(X_i)} g(x) dP(x),$$
so that 
$$ \Pi_T^*(g)=  n\cdot \biggl(\int_{T^{-1}(X_1)} g(x) dP(x), \dots, \int_{T^{-1}(X_n)} g(x) dP(x) \biggr) $$
is the adjoint operator of $\Pi_T$. Furthermore, 
\begin{align*}
    \int \bigl((\Pi_T f)(x)\bigr)^2dP(x)&= \int \bigl(f(T(x))\bigr)^2dP(x)= \int \bigl(f(x)\bigr)^2d\widehat{P}(x),
\end{align*}
so that $\Pi_T$ is an isometry. 
\end{proof}
Next, we compare the coercivity of $\widehat{\Theta}_\eps$ with that of 
$$ {\Sigma}_{\eps,T}(u,v) = \int \int_{\mathbb{B}_{\frac{\delta_\psi\eps}{16 L}}(x)} \bigl( (\Pi_T u)(x)-(\Pi_T u)(z) \bigr)\bigl( (\Pi_T v)(x)-(\Pi_T v)(z) \bigr) dP(z) dP(x)  $$
under the event $\mathcal{E}_n$.
\begin{lemma}\label{lemma:Theta-sigma}
Assume that there exists ${T}$ such that ${T}_\# P =\widehat{P}$ and $  \|{T}-{\rm I}\|_{L^\infty(P)} \leq \frac{\delta_\psi\eps}{32 L}$. Then 
$$ {\Sigma}_{\eps,T}(u,u)\leq   \widehat{\Theta}_\eps(u,u). $$
\end{lemma}
\begin{proof}
Using the push-forward condition, 
\begin{align*}
    \widehat{\Theta}_\eps(u,u) &= \frac{1}{n^2} \sum_{\|X_i-X_\ell\| \leq \frac{\delta_\psi \eps}{8L}}     \bigl(u(X_i)- u(X_\ell) \bigr)^2\\
   &=  \int {\bf 1}_{\|T(x)-T(z)\| \leq \frac{\delta_\psi\eps}{8 L} }    \bigl(u(T(x))- u(T(z)) \bigr)^2 dP(x) dP(z)\\
   &\geq \int {\bf 1}_{\|x-z\| \leq \frac{\delta_\psi\eps}{16 L} }    \bigl(u(T(x))- u(T(z)) \bigr)^2 dP(x) dP(z)\\
   &=  {\Sigma}_{\eps,T}(u,u). \qedhere
\end{align*}

\end{proof}
The following technical
result, which allows us to control 
$\Sigma_{\eps, \hat T}(u, u)$ via the population variance
of $u \circ \hat T$,
is taken from~\cite[Lemma 3.3]{GonzalezSanzNutzRiveros.26}.

\begin{lemma}[\cite{GonzalezSanzNutzRiveros.26}]\label{Lemma:Bound-energy} Assume that $\Omega$ is convex.
      For every $h\in L^2(P)$ and $\delta>0$, 
       $$ \int \int \bigl(h(x)-h(z)\bigr)^2   d P(x) dP(z) \leq  \frac{\Lambda_P^2}{\lambda_P^2}\biggl\lceil\frac{{\rm diam}(\Omega)}{\delta } \biggr\rceil^{d+2} \int \int_{\mathbb{B}_\delta(x) } \bigl(h(x)-h(z)\bigr)^2  d P(x) dP(z).   $$
   \end{lemma} 
\begin{proof}[Proof of \cref{pr:comparison-Theta-Var}]
Assuming that there exists a map ${T}$ such that ${T}_\# P =\widehat{P}$ and $  \|{T}-{\rm I}\|_{L^\infty(P)} \leq {\frac{\delta_\psi\eps}{32 L}}$, 
\Cref{lemma:Isometry,lemma:Theta-sigma,Lemma:Bound-energy}  yield
\begin{align}
    \begin{split}
        \label{eq:Bound-under-En}
        \widehat{\Theta}_\eps(u,u)\geq  {\Sigma}_{\eps,T}(u,u)&\geq  {\frac{2\lambda_P^2}{\Lambda_P^2\Bigl(\Bigl\lceil\frac{16 L{\rm diam}(\Omega)}{\delta_\psi\eps} \Bigr\rceil \Bigr)^{d+2}}} \min_{a\in \R}\|\Pi_T(u)-a\|_{L^2(P)}^2 \\
        &\geq {\frac{\lambda_P^2}{\Lambda_P^2\Bigl(\Bigl\lceil\frac{16 L{\rm diam}(\Omega)}{\delta_\psi\eps} \Bigr\rceil \Bigr)^{d+2}}} \min_{a\in \R}\|u-a\|_{L^2(\widehat{P})}^2,
    \end{split}
\end{align}
which concludes the proof.     
\end{proof}

We note that \cref{proposition:component-coercivity} follows from
\cref{Lemma:From two-to-one-variance,pr:comparison-Theta-Var} and the
definition of $\gamma(n,\eps)$ in \eqref{eq:gamma}.

\subsection{Coercivity of the random bilinear form}\label{se:coercivity-joint}
Next, we show the coercivity of $\widehat{B}_r$ under the event $\mathcal{E}_n$. The idea of the proof is to show that $\widehat{B}_r$ defines a nonnegative operator $ \widehat{\mathbb{M}}_r$ which admits an eigenvalue decomposition. Then, we bound the eigenvalues on the strength of \Cref{proposition:component-coercivity}. 
\begin{proposition}\label{proposition:coercivity}
For $0\leq r\leq t_\eps$, 
we have
    $$ \widehat{B}_r( u\oplus v,u\oplus v )  \geq \widehat{\beta} \| u\oplus v \|_{L^2(\widehat{P}\otimes\widehat{Q})}^2{\bf 1}_{\mathcal{E}_n},$$
    where
   $$  \widehat{\beta}= \min\biggl(\frac{ \widehat{\alpha} \psi''\bigl(t_\psi-\delta_\psi\bigr) \min_{i}  \widehat{P}(\mathbb{B}_{\frac{\delta_\psi \eps}{8L}} (X_i))}{ \psi''\bigl(t_\psi-\delta_\psi\bigr) \min_{i}  \widehat{P}(\mathbb{B}_{\frac{\delta_\psi \eps}{8L}} (X_i))+ 4\gamma(n,\eps)},  \frac{1}{2}\psi''\bigl(t_\psi-\delta_\psi\bigr) \min_{i}  \widehat{P}(\mathbb{B}_{\frac{\delta_\psi \eps}{8L}}(X_i)) \biggr).  $$
    As a consequence,
$$ \int  \Bigl(\bigl( \widehat{f} - f_\eps \bigr) \oplus \bigl( \widehat{g} - g_\eps \bigr) \Bigr)^2\psi''(\gamma_r)  d(\widehat{P}\otimes \widehat{Q}) \geq \frac{\widehat{\beta}}{2} \int  \Bigl(\bigl( \widehat{f} - f_\eps \bigr) \oplus \bigl( \widehat{g} - g_\eps \bigr) \Bigr)^2  d(\widehat{P}\otimes \widehat{Q}) {\bf 1}_{\mathcal{E}_n} . $$
\end{proposition}
\begin{proof}

The bilinear form
$$   \widehat{B}_r( u\oplus v,\tilde u\oplus \tilde v ) = \int \psi''(\gamma_r) \bigl(u\oplus v\bigr)  \bigl(\tilde u\oplus \tilde v\bigr)   d(\widehat{P}\otimes \widehat{Q}) $$
is symmetric and nonnegative. Hence, there exists an operator
$\widehat{\mathbb{M}}_r:\widehat{L}_\oplus^2 \to \widehat{L}_\oplus^2 $
such that 
$$  \widehat{B}_r( u\oplus v,\tilde u\oplus \tilde v ) = \bigl\langle \widehat{\mathbb{M}}_r\bigl(u\oplus v\bigr), \tilde u\oplus \tilde v \bigr\rangle_{L^2(\widehat{P}\otimes \widehat{Q})},$$
for all $u\oplus  v, \tilde u\oplus \tilde v \in \widehat{L}_\oplus^2$. The operator $\widehat{\mathbb{M}}_r$ admits an eigenvalue decomposition, and all eigenvalues are nonnegative since $\widehat{B}_r$ is symmetric and nonnegative.  Let $\lambda\geq 0$ be a nontrivial eigenvalue with associated eigenvector $ u_0\oplus  v_0\neq 0$. Assume furthermore that  $ \mathcal{E}_n$ holds. We can assume that $u_0$ is centered with respect to $\widehat{P}$, as this does not change the value of  $ u_0\oplus  v_0 $. 

Fixing $u\in L^2(\widehat{P})$ with $\int u d\widehat{P}=0$, we have
\begin{align*}
    \lambda \bigl\langle u_0, u\bigr\rangle_{L^2(\widehat{P})}&=  \lambda \bigl\langle u_0\oplus v_0, u\oplus 0 \bigr\rangle_{L^2(\widehat{P}\otimes \widehat{Q})}\\
    &=  \bigl\langle \widehat{\mathbb{M}}_r\bigl(u_0\oplus v_0\bigr),  u\oplus 0  \bigr\rangle_{L^2(\widehat{P}\otimes \widehat{Q})}\\
    &= \int u(x)\psi''(\gamma_r(x,y)) (u_0(x)  +v_0(y))     d(\widehat{P}\otimes \widehat{Q})(x,y)\\
    &  = \int  u(x) \biggl(u_0(x) \int \psi''(\gamma_r(x,y))  d\widehat{Q}(y)  + \int \psi''(\gamma_r(x,y))v_0(y) d\widehat{Q}(y)     \biggr) d\widehat{P}(x).
\end{align*}
Since $u\in L^2(\widehat{P})$ with $\int u d\widehat{P}=0$ was arbitrary and $u_0$ is  centered, we get that, for every $i\in \{1, \dots, n\}$,
\begin{multline*}
    \lambda u_0(X_i)= u_0(X_i) \int \psi''(\gamma_r(X_i,y))  d\widehat{Q}(y)  + \int \psi''(\gamma_r(X_i,y))v_0(y) d\widehat{Q}(y)\\- \int  \biggl(u_0(x) \int \psi''(\gamma_r(x,y))  d\widehat{Q}(y)  + \int \psi''(\gamma_r(x,y))v_0(y) d\widehat{Q}(y)     \biggr) d\widehat{P}(x).
\end{multline*}
Furthermore, fixing $v\in L^2(\widehat{Q})$ and using that $\int u_0 d\widehat{P}=0$,  we derive the identity $$\lambda \langle v_0, v\rangle_{L^2(\widehat{Q})}=  \lambda \langle u_0\oplus v_0,0\oplus v \rangle_{L^2(\widehat{P}\otimes \widehat{Q})}.$$ Therefore, repeating the same argument as before, 
$$ \lambda v_0(Y_j)= v_0(Y_j) \int \psi''(\gamma_r(x,Y_j))  d\widehat{P}(x)  + \int \psi''(\gamma_r(x,Y_j))u_0(x) d\widehat{P}(x),$$
and in particular 
$$ v_0(Y_j)= \frac{\int \psi''(\gamma_r(x,Y_j))u_0(x) d\widehat{P}(x)}{\lambda-\int \psi''(\gamma_r(x,Y_j)) d\widehat{P}(x)} , $$
whenever the denominator is
nonzero. 
There are two cases, 
$$  0\leq \lambda\leq \frac{1}{2} \min_j\int \psi''(\gamma_r(x,Y_j)) d\widehat{P}(x) \quad \text{or}\quad \lambda> \min_j\frac{1}{2}\int \psi''(\gamma_r(x,Y_j)) d\widehat{P}(x) ,$$
which we analyze separately.
    
 {\it Case 1: $\lambda> \frac{1}{2}\min_j\int \psi''(\gamma_r(x,Y_j)) d\widehat{P}(x)$.}   We use the fact that $\psi''$ is monotone and $0\leq r\leq t_\eps$ to see that for some $j\in \{1, \dots, n\}$, $$\lambda> \frac{1}{2}\int \psi''(\gamma_r(x,Y_j)) d\widehat{P}(x)$$ and $$\int \psi''(\gamma_r(x,Y_j)) d\widehat{P}(x)\geq \int \psi''\biggl(\frac{\widehat{f}(x)+ \widehat{g}(Y_j)-c(x,Y_j) }{\eps}-\frac{\delta_\psi}{4}\biggr)d\widehat{P}(x).$$ 
  Arguing as in the proof of \Cref{Lemma:From two-to-one-variance}, it follows that 
$$ \int \psi''\biggl(\frac{\widehat{f}(x)+ \widehat{g}(Y_j)-c(x,Y_j) }{\eps}-\frac{\delta_\psi}{4}\biggr) d\widehat{P}(x)\geq  \psi''\bigl(t_\psi- \delta_\psi\bigr) \frac{\#\bigl\{ i: \|X_{i}-X_{i(j)}\|\leq \frac{\delta_\psi \eps}{8L} \bigr\}}{n},   $$
where 
$$i(j)\in \argmax_{i} \{\widehat{f}(X_i)+ \widehat{g}(Y_j)-c(X_i,Y_j)\}.$$
This implies 
\begin{align}\label{eq:bound-case-1-lambda}
   \int \psi''\biggl(\frac{\widehat f(x)+\widehat g(Y_j)-c(x,Y_j)}{\eps}
-\frac{\delta_\psi}{4}\biggr)\,d\widehat P(x)
\geq
\psi''(t_\psi-\delta_\psi)\min_i \widehat P\Bigl(\mathbb B_{\frac{\delta_\psi\eps}{8L}}(X_i)\Bigr),
\end{align}
which yields 
\begin{align*}
     \lambda\geq \frac{1}{2} \int \psi''\biggl(\frac{\widehat{f}(x)+ \widehat{g}(Y_j)-c(x,Y_j) }{\eps}-\frac{\delta_\psi}{4}\biggr) d\widehat{P}(x)\geq  \frac{1}{2}\psi''\bigl(t_\psi- \delta_\psi\bigr) \min_{i}  \widehat{P}\Bigl(\mathbb{B}_{\frac{\delta_\psi \eps}{8L}}(X_i)\Bigr).
\end{align*}
In this case, the proof is complete here. 

{\it Case 2: $0\leq \lambda\leq  \frac{1}{2}\min_j\int \psi''(\gamma_r(x,Y_j)) d\widehat{P}(x)$.}  By Jensen's inequality, 
    \begin{align*}
        \bigl(v_0(Y_j)\bigr)^2&=  \biggl(\frac{\int \psi''(\gamma_r(x,Y_j))u_0(x) d\widehat{P}(x)}{\lambda-\int \psi''(\gamma_r(x,Y_j)) d\widehat{P}(x)} \biggr)^2 \\&\leq 4\biggl(\frac{\int \psi''(\gamma_r(x,Y_j))u_0(x) d\widehat{P}(x)}{\int \psi''(\gamma_r(x,Y_j)) d\widehat{P}(x)} \biggr)^2 
        \leq 4 \frac{\int \psi''(\gamma_r(x,Y_j))\bigl(u_0(x)\bigr)^2 d\widehat{P}(x)}{\int \psi''(\gamma_r(x,Y_j)) d\widehat{P}(x)} .
    \end{align*}
We use the fact that $\psi''$ is monotone and $0\leq r\leq t_\eps$ to find
 \begin{align*}
        \bigl(v_0(Y_j)\bigr)^2  \leq 4 \frac{\int \psi''\Bigl(\frac{\widehat{f}(x)+ \widehat{g}(Y_j)-c(x,Y_j) }{\eps}+\frac{\delta_\psi}{4}\Bigr)\bigl(u_0(x)\bigr)^2 d\widehat{P}(x)}{\int \psi''\Bigl(\frac{\widehat{f}(x)+ \widehat{g}(Y_j)-c(x,Y_j) }{\eps}-\frac{\delta_\psi}{4}\Bigr) d\widehat{P}(x)} .
    \end{align*}
Applying \eqref{eq:bound-case-1-lambda}, we obtain
\begin{align}\label{eq:bound-case-2-lambda}
        \|v_0\|^2_{L^2(\widehat{Q})} & \leq 4 \frac{\int u_0^2 \int \psi''\Bigl(\frac{\widehat{f}\oplus \widehat{g}-c }{\eps}+\frac{\delta_\psi}{4}\Bigr) d\widehat{Q}(y) d\widehat{P}(x)}{\psi''\bigl(t_\psi-\delta_\psi\bigr) \min_{i}  \widehat{P}(\mathbb{B}_{\frac{\delta_\psi \eps}{8L}}(X_i))} \notag \\
        &\le \frac{4\gamma(n,\eps)  }{\psi''\bigl(t_\psi-\delta_\psi\bigr) \min_{i}  \widehat{P}(\mathbb{B}_{\frac{\delta_\psi \eps}{8L}} (X_i))}      \|u_0\|^2_{L^2(\widehat{P})} .
    \end{align}
Note that this implies that $\|u_0\|^2_{L^2(\widehat{P})}\neq 0$.  
Since $u_0\oplus v_0$ is an eigenvector of $\widehat{\mathbb{M}}_r$ with eigenvalue $\lambda$,  it follows that 
 \begin{align*}
     \lambda \| u_0\oplus v_0\|_{L^2(\widehat{P}\otimes \widehat{Q})}^2  &=\bigl\langle u_0\oplus v_0,  \widehat{\mathbb{M}}_r\bigl(u_0\oplus v_0\bigr)\bigr\rangle_{L^2(\widehat{P}\otimes \widehat{Q})}\geq \widehat{\alpha}{\|u_0\|_{L^2(\widehat{P})}^2} {\bf 1}_{\mathcal{E}_n},
 \end{align*}
 where in the last inequality we used \Cref{proposition:component-coercivity} and the fact that $u_0$ is centered. Using again that $u_0$ is centered, we get  $\| u_0\oplus v_0\|_{L^2(\widehat{P}\otimes \widehat{Q})}^2= \|u_0\|_{L^2(\widehat{P})}^2+ \|v_0\|_{L^2(\widehat{Q})}^2 $, which in turn yields 
 \begin{equation}
     \label{eq:lambda-larger-than-alpha}
 \lambda \geq \widehat{\alpha} \frac{\|u_0\|_{L^2(\widehat P)}^2}{\|u_0\|_{L^2(\widehat P)}^2+ \|v_0\|_{L^2(\widehat Q)}^2} {\bf 1}_{\mathcal{E}_n} \, .\end{equation}
 Plugging \eqref{eq:bound-case-2-lambda} into \eqref{eq:lambda-larger-than-alpha}, we derive 
 $$ \lambda \geq \frac{ \widehat{\alpha} \psi''\bigl(t_\psi-\delta_\psi\bigr) \min_{i}  \widehat{P}(\mathbb{B}_{\frac{\delta_\psi \eps}{8L}} (X_i))}{ \psi''\bigl(t_\psi-\delta_\psi\bigr) \min_{i}  \widehat{P}(\mathbb{B}_{\frac{\delta_\psi \eps}{8L}} (X_i))+ 4\gamma(n,\eps)} {\bf 1}_{\mathcal{E}_n} $$
 for any eigenvalue $\lambda$ of $\widehat{\mathbb{M}}_r$ under Case 2. Combining this with the lower bound obtained in Case 1 yields \(\lambda\geq \widehat\beta {\bf 1}_{\mathcal E_n}\) for every eigenvalue \(\lambda\).  This completes the proof.   
\end{proof}

\proofreadhere

\subsection{Replacing random bounds by deterministic ones}\label{se:deterministic-bounds}
There are three random terms in \Cref{proposition:coercivity}:  
$\min_{i}  \widehat{P}\Bigl(\mathbb{B}_{\frac{\delta_\psi \eps}{8L}}(X_i)\Bigr)$, $\min_{j}  \widehat{Q}\Bigl(\mathbb{B}_{\frac{\delta_\psi \eps}{8L}}(Y_j)\Bigr)$, and ${\bf 1}_{\mathcal{E}_n}$.  In this subsection we provide bounds for these quantities. 

\subsubsection{Probability of empirical balls}
The following result shows that the empirical measures of the balls are close to their population counterparts. In particular, we lower bound the probability of the events
$$ \mathcal{F}_{n}=\biggl\{\min_{j}  \widehat{Q}\Bigl(\mathbb{B}_{\frac{\delta_\psi \eps}{8L}}(Y_j)\Bigr)\geq \frac{1}{2} \min_{y\in \Omega'} {Q}\Bigl(\mathbb{B}_{\frac{\delta_\psi \eps}{8L}}(y)\Bigr) \biggr\}$$
and 
$$ \mathcal{G}_{n}=\biggl\{\min_{i}  \widehat{P}(\mathbb{B}_{\frac{\delta_\psi \eps}{8L}}(X_i))\geq \frac{1}{2} \min_{x\in \Omega} P\Bigl(\mathbb{B}_{\frac{\delta_\psi \eps}{8L}} (x)\Bigr) \biggr\}.$$
The proof, which is standard, follows from VC-theory and Dudley’s integral entropy bound.   
\begin{proposition}\label{prop-VC}
There exists a constant $C\geq 0$, independent of $\eps$, such that 
   $$ \mathbb{P}\biggl( \min_{j}  \widehat{Q}\Bigl(\mathbb{B}_{\frac{\delta_\psi \eps}{8L}}(Y_j)\Bigr)\geq \frac{1}{2}\min_{y\in \Omega'} {Q}\Bigl(\mathbb{B}_{\frac{\delta_\psi \eps}{8L}}(y)\Bigr)\biggr) \geq 1-\frac{C}{n\,\Bigl(\min_{y\in \Omega'} {Q}\Bigl(\mathbb{B}_{\frac{\delta_\psi \eps}{8L}}(y)\Bigr) \Bigr)^2} $$
and 
   $$ \mathbb{P}\biggl( \min_{i}  \widehat{P}\Bigl(\mathbb{B}_{\frac{\delta_\psi \eps}{8L}}(X_i)\Bigr)\geq \frac{1}{2}\min_{x\in \Omega} P\Bigl(\mathbb{B}_{\frac{\delta_\psi \eps}{8L}} (x)\Bigr)\biggr) \geq 1-\frac{C}{n\,\Bigl(\min_{x\in \Omega } P\Bigl(\mathbb{B}_{\frac{\delta_\psi \eps}{8L}} (x)\Bigr) \Bigr)^2}.  $$
\end{proposition}
\begin{proof}
The collection $\mathcal{B}$ of all open balls  in  $\R^d$ is a VC-class of order $d+2$ by  \cite[Lemma~2.6.15]{vanderVaart.1996}. Hence, by Theorem~2.6.4, ibid, there exists a dimensional constant $C$ such that,  for any probability measure $\mu$ and any $\delta \in (0,1)$,  
$$ \mathcal{N}( \mathcal{B}, L^2(\mu), \delta ) \leq \frac{C}{\delta^{2(d+1)}} . $$
Here $\mathcal{N}( \mathcal{B}, L^2(\mu), \delta ) $ denotes the minimal number of balls of radius $\delta>0$ in $ L^2(\mu)$ necessary to cover $\{{\bf 1}_{B}: B\in \mathcal{B}\}$. Hence, the bound 
\begin{equation}\label{eq:metric-entropy-Bound}
    \int_{0}^1 \sup_{\mu} \sqrt{\log\bigl( \mathcal{N}( \mathcal{B}, L^2(\mu), \delta )  \bigr)} d\delta \leq C
\end{equation}
on the metric entropy holds
for a possibly different dimensional constant $C$.  Define the stochastic process 
$$ \{{\bf 1}_{B}: B\in \mathcal{B}\}\ni   {\bf 1}_B\mapsto \mathbb{Q}_n ({\bf 1}_B) =  {Q}({\bf 1}_B)-\widehat{Q}({\bf 1}_B)$$ 
over the metric space $(\{{\bf 1}_{B}: B\in \mathcal{B}\}, L^2(Q))$. We define the symmetrized version of the process as 
$$ \{{\bf 1}_{B}: B\in \mathcal{B}\}\ni   {\bf 1}_B\mapsto \mathbb{Q}_n^S ({\bf 1}_B)= \frac{1}{n}\sum_{j=1}^n \epsilon_j  {\bf 1}_B(Y_j),$$ 
where $\{\epsilon_i\}_i$  is an i.i.d.~sequence of Rademacher random variables independent of $Y_1, \dots, Y_n$. First note that the random variable
$ \sup_{B\in {\mathcal{B}}}\mathbb{Q}_n ({\bf 1}_B)  $ can be written as $ \sup_{B\in \widetilde{\mathcal{B}}}\mathbb{Q}_n ({\bf 1}_B)  $, where  $\widetilde{\mathcal{B}}$ denotes the collection of all open balls  in  $\R^d$ with rational radii and centers. (This avoids measurability issues in the sequel.)  The symmetrization lemma yields
\begin{equation}
    \label{eq:symmetrization}\EE\bigl[ \bigl(\sup_{B\in {\mathcal{B}}}\mathbb{Q}_n ({\bf 1}_B) \bigr)^2\bigr] \leq 4\, \EE\bigl[ \bigl(\sup_{B\in \widetilde{\mathcal{B}}}\mathbb{Q}_n^S ({\bf 1}_B) \bigr)^2\bigr].
\end{equation}
Hence, by Hoeffding's inequality, the process $\sqrt{n}\,\mathbb{Q}_n^S$ satisfies 
$$ \PP\biggl( \sqrt{n}  \bigl|\mathbb{Q}_n^S ({\bf 1}_{B_1})- \mathbb{Q}_n^S ({\bf 1}_{B_2}) \bigr|\geq s\, \bigm\vert Y_1, \dots Y_n\biggr) \leq 2 e^{-\frac{s^2}{2 \|{\bf 1}_{B_2}-{\bf 1}_{B_1}\|^2_{L^2(\widehat{Q})} }},$$
i.e., $ \sqrt{n} \mathbb{Q}_n^S | Y_1, \dots Y_n $ is sub-Gaussian with respect to the  $L^2(\widehat{Q})$-norm. Hence, Dudley’s integral entropy bound (cf.~\cite[Corollary~2.2.8]{vanderVaart.1996})  yields 
$$n \, \EE\biggl[ \bigl(\sup_{B\in \widetilde{\mathcal{B}}}\mathbb{Q}_n^S ({\bf 1}_B)  \bigr)^2 \bigm\vert Y_1, \dots Y_n\biggr]  \lesssim \biggl( \int_{0}^1 \sqrt{\log\bigl( \mathcal{N}( \mathcal{B}, L^2(\widehat{Q}), \delta )  \bigr)} d\delta \biggr)^2  \leq C^2, $$
where the final constant is finite and deterministic by \eqref{eq:metric-entropy-Bound}. Taking expectations and using~\eqref{eq:symmetrization}, we conclude that 
$n \, \EE\big[ \big(\sup_{B\in {\mathcal{B}}}\mathbb{Q}_n ({\bf 1}_B)  \big)^2\big]=\mathcal{O}(1)$. Then
\begin{align*}
    \mathbb{P}(\mathcal{F}_n)&\geq \mathbb{P}\biggl( \min_{y\in \Omega'}  \widehat{Q}\Bigl(\mathbb{B}_{\frac{\delta_\psi \eps}{8L}}(y)\Bigr) \geq \frac{1}{2}\min_{y\in \Omega'} {Q}\Bigl(\mathbb{B}_{\frac{\delta_\psi \eps}{8L}}(y)\Bigr)\biggr) \\
    &\geqslant  \mathbb{P}\biggl( \sup_{y\in \Omega'} \mathbb{Q}_n\Bigl(\mathbb{B}_{\frac{\delta_\psi \eps}{8L}}(y)\Bigr) \leq \frac{1}{2}\min_{y\in \Omega'} {Q}\Bigl(\mathbb{B}_{\frac{\delta_\psi \eps}{8L}}(y)\Bigr)\biggr)\\
    &=1-\mathbb{P}\biggl( \sup_{y\in \Omega'}  \mathbb{Q}_n\Bigl(\mathbb{B}_{\frac{\delta_\psi \eps}{8L}}(y)\Bigr) >  \frac{1}{2}\min_{y\in \Omega'} {Q}\Bigl(\mathbb{B}_{\frac{\delta_\psi \eps}{8L}}(y)\Bigr)\biggr)\\
    &\geq 1- \frac{4\EE\bigl[ \bigl( \sup_{B\in {\mathcal{B}}}\mathbb{Q}_n ({\bf 1}_B) \bigr)^2 \bigr]}{ \Bigl(\min_{y\in \Omega'} {Q}\Bigl(\mathbb{B}_{\frac{\delta_\psi \eps}{8L}}(y)\Bigr) \Bigr)^2},
\end{align*}
completing the proof. The estimate for \(\mathcal G_n\) is obtained by the identical argument with \(P,X_i,\Omega\) in place of \(Q,Y_j,\Omega'\).
\end{proof}

\subsubsection{The $\infty$-Wasserstein distance}

In this section we recall some of the results from \cite{Trillos-Wass-infty} regarding the empirical  $\infty$-Wasserstein distance, defined as 
 $$ \widehat{W}_\infty = \inf_{\pi \in \Pi(P, \widehat P)} \bigl\|x - y\bigr\|_{L^\infty(\pi)}.$$
Because $P$ is absolutely continuous with respect to Lebesgue measure, the infimum is attained by a map, in the sense that for each $n\in \NN$ there exists a transport map $\hat{T}$ with $\hat{T}_{\#} P=\widehat{P}$ and 
 $$ \widehat{W}_\infty =  \|\hat{T}-I\|_{L^\infty(P)}.$$
As a consequence, we have the equality 
$$\mathcal{E}_n= \Bigl\{   \widehat{W}_\infty \leq {\frac{\delta_\psi\eps}{32 L}}\Bigr\}. $$
We define the sequence
$$ \alpha_{n,d}=\begin{cases}
  (\log(n))^{\frac{1}{2}}  {n^{-\frac{1}{2}}} & \text{if } d=1,\\
{(\log(n))^{\frac{3}{4}}}{n^{-\frac{1}{2}}} & \text{if } d=2,\\
{(\log(n))^{\frac{1}{d}}}{n^{-\frac{1}{d}}} & \text{if } d>2.
\end{cases}$$
\begin{theorem}[\cite{Trillos-Wass-infty} and \cite{Liu.Liu.Lu.2019.Quart}]\label{theo-Wasserstein_infty}
   There exist constants $C,A\geq 0$ such that for all $n\in \NN$, 
   $$ \mathbb{P}\bigl( \widehat{W}_\infty\geq A\alpha_{n,d}\bigr) \leq \frac{C}{n^2}.$$

\end{theorem}
\begin{corollary}\label{corollary:Nicolas}
Fix ${u>\max(d,2)}$. Then 
there exists $n_0\in \NN$ such that for all $n\geq n_0$, 
$$ \EE\bigl[\widehat{W}_\infty^{{u}}\bigr] \leq {\frac{\delta_\psi^{{u}}}{n (32 L)^{{u}}}}.$$
As a consequence, for any ${u>\max(d,2)}$  there exists  $n_0\in \NN$ such that for all $n\geq n_0$,  
$$\mathbb{P}(\mathcal{E}_n^c) \leq \frac{1}{\eps^{{u}} n }.$$
\end{corollary}
\begin{proof}
  Since $\widehat{W}_\infty\leq {\rm diam}(\Omega
  )$ and for $n$ sufficiently large $A\alpha_{n,d}\leq {\rm diam}(\Omega)$, then 
    \begin{align*}
        \EE\bigl[\widehat{W}_\infty^u\bigr]&= \int_{0}^\infty \mathbb{P}\bigl( \widehat{W}_\infty\geq \lambda^{\frac{1}{u}} \bigr)  d\lambda\\
        &=\int_{0}^{\bigl({\rm diam}(\Omega)\bigr)^u} \mathbb{P}\bigl( \widehat{W}_\infty\geq \lambda^{\frac{1}{u}} \bigr)  d\lambda \\
        &=  \int_{0}^{\bigl(A \alpha_{n,d}\bigr)^u} \mathbb{P}\bigl( \widehat{W}_\infty\geq \lambda^{\frac{1}{u}} \bigr)  d\lambda +  \int_{\bigl(A \alpha_{n,d}\bigr)^u}^{\bigl({\rm diam}(\Omega)\bigr)^u}\mathbb{P}\bigl( \widehat{W}_\infty\geq \lambda^{\frac{1}{u}} \bigr)  d\lambda\\
         &\leq \bigl(A \alpha_{n,d}\bigr)^u +  \bigl({\rm diam}(\Omega)\bigr)^u \mathbb{P}\bigl( \widehat{W}_\infty\geq A \alpha_{n,d} \bigr)  \leq \frac{\delta_\psi^u}{n (32 L)^u}
    ,\end{align*}
where the last inequality holds for $n$ sufficiently large, by \Cref{theo-Wasserstein_infty}. Hence, the first claim holds. The second is derived from the first via Markov's inequality. 
\end{proof}
\subsection{Conclusion of the proof of \cref{theorem:potentials}}

\begin{proof}[Proof of \cref{theorem:potentials}]
We recall that
$$  \beta_{n,\eps}= \min\Biggl(\frac{ 4{\alpha}_{n,\eps} \psi''\bigl(t_\psi-\delta_\psi\bigr) \min_{x\in \Omega}  P\Bigl(\mathbb{B}_{\frac{\delta_\psi \eps}{8L}} (x)\Bigr)}{ \psi''\bigl(t_\psi-\delta_\psi\bigr) \min_{x\in \Omega}  P\Bigl(\mathbb{B}_{\frac{\delta_\psi \eps}{8L}} (x)\Bigr)+ 8\gamma(n,\eps)}, \frac{1}{2}\psi''\bigl(t_\psi-\delta_\psi\bigr) \min_{x\in \Omega}  P\Bigl(\mathbb{B}_{\frac{\delta_\psi \eps}{8L}} (x)\Bigr) \Biggr),  $$
where
$$ \alpha_{n,\eps}=   \frac{ \lambda_P^2 \psi''\bigl(t_\psi -\delta_\psi  \bigr)^2  \min_{y\in \Omega' }  {Q}\Bigl(\mathbb{B}_{\frac{\delta_\psi \eps}{8L}}(y)\Bigr)}{8\Lambda_P^2 \gamma(n,\eps) \Bigl\lceil\frac{16 L{\rm diam}(\Omega)}{\delta_\psi\eps } \Bigr\rceil^{d+2}}. $$
By the definitions of $\mathcal{F}_n$ and $ \mathcal{G}_n$,
$$ \widehat{\beta}\mathbf{1}_{\mathcal{E}_n \cap \mathcal{F}_n \cap \mathcal{G}_n} \geq \frac{\beta_{n,\eps}}{2}\mathbf{1}_{\mathcal{E}_n \cap \mathcal{F}_n \cap \mathcal{G}_n} . $$
To simplify the notation, we set
$$ A_n = \int  \Bigl(\bigl( \widehat{f} - f_\eps \bigr) \oplus \bigl( \widehat{g} - g_\eps \bigr) \Bigr)^2  d(\widehat{P}\otimes \widehat{Q}) .$$
Hence, \Cref{proposition:coercivity} yields 
\begin{align*}
    \Gamma(t_\eps) -\Gamma(0) &\geq      \frac{t_\eps^2\, \beta_{n,\eps} \, A_n}{ 8 \eps}  \mathbf{1}_{\mathcal{E}_n \cap \mathcal{F}_n \cap \mathcal{G}_n} = \frac{t_\eps^2\, \beta_{n,\eps}  A_n }{8 \eps}  \bigl(1- \mathbf{1}_{\mathcal{E}_n^c \cup \mathcal{F}_n^c\cup \mathcal{G}_n^c}\bigr),
\end{align*}
which implies 
\begin{align*}
   \frac{t_\eps^2\, \beta_{n,\eps}  A_n }{8\eps} &\leq    \frac{t_\eps^2\beta_{n,\eps}}{4\eps}\bigl(  \|f_\eps\oplus g_\eps\|_\infty^2  +  \|\widehat{f}\oplus \widehat{g}\|_\infty^2  \bigr) \mathbf{1}_{\mathcal{E}_n^c \cup \mathcal{F}_n^c \cup \mathcal{G}_n^c} + \Gamma(t_\eps)-\Gamma(0)\\
\text{(\Cref{pr:ROTprelims}) }&\leq   \frac{t_\eps^2\beta_{n,\eps}}{2\eps} \bigl(5 \|c\|_\infty + \eps \phi'(1)\bigr)^2 \mathbf{1}_{\mathcal{E}_n^c \cup \mathcal{F}_n^c \cup \mathcal{G}_n^c}+ \Gamma(t_\eps)-\Gamma(0) \\
\text{(convexity of $\Gamma$) }&\leq \frac{t_\eps^2\beta_{n,\eps}}{2\eps} \bigl(5 \|c\|_\infty + \eps \phi'(1)\bigr)^2\mathbf{1}_{\mathcal{E}_n^c \cup \mathcal{F}_n^c \cup \mathcal{G}_n^c} + t_\eps\Gamma' (t_\eps)\\
\text{(monotonicity of $\Gamma'$) } &\leq \frac{t_\eps^2\beta_{n,\eps}}{2\eps} \bigl(5 \|c\|_\infty + \eps \phi'(1)\bigr)^2\mathbf{1}_{\mathcal{E}_n^c \cup \mathcal{F}_n^c \cup \mathcal{G}_n^c}+ t_\eps\Gamma' (1).
\end{align*} 
Taking expectations, we derive the bound 
\begin{align*}
   \frac{t_\eps^2\, \beta_{n,\eps}  \EE\bigl[A_n\bigr] }{8\eps} &\leq \frac{t_\eps^2\beta_{n,\eps}}{2\eps} \bigl(5 \|c\|_\infty + \eps \phi'(1)\bigr)^2\PP\bigl({\mathcal{E}_n^c \cup \mathcal{F}_n^c \cup  \mathcal{G}_n^c}\bigr) + t_\eps \EE\bigl[\Gamma' (1)\bigr]\\
    &\leq \frac{t_\eps^2\beta_{n,\eps}}{2\eps} \bigl(5 \|c\|_\infty + \eps \phi'(1)\bigr)^2 \bigl(\PP(\mathcal{E}_n^c) + \PP(\mathcal{F}_n^c)+\PP(\mathcal{G}_n^c)\bigr) + t_\eps \EE\bigl[\Gamma' (1)\bigr].
\end{align*}
We fix $u>\max(d,2)$ and apply \Cref{corollary:Nicolas} and \Cref{prop-VC} to get 
\begin{multline*}
     \frac{t_\eps^2\, \beta_{n,\eps}  \EE\bigl[A_n\bigr] }{8\eps} \\ \leq \frac{t_\eps^2\beta_{n,\eps}}{2\eps} \bigl(5 \|c\|_\infty + \eps \phi'(1)\bigr)^2 \biggl(\frac{1}{\eps^u n} + \frac{C}{n  \Bigl(\min_{x\in \Omega}  P\Bigl(\mathbb{B}_{\frac{\delta_\psi \eps}{8L}} (x)\Bigr) \Bigr)^2 } + \frac{C}{n  \Bigl(\min_{y\in \Omega'}  {Q}\Bigl(\mathbb{B}_{\frac{\delta_\psi \eps}{8L}} (y)\Bigr) \Bigr)^2 } \biggr)\\+ t_\eps \EE\bigl[\Gamma' (1)\bigr].
\end{multline*}
Note that 
$$\EE\bigl[ \Gamma' (1)\bigr] = \EE\biggl[ \int \bigl( \widehat{f} - f_\eps \bigr) \oplus \bigl( \widehat{g} - g_\eps \bigr) 
\bigl( 1 - \psi'\biggl(\frac{f_\eps \oplus g_\eps - c}{\eps}\biggr) \bigr) \, d\widehat{P} \otimes \widehat{Q} \biggr] ,$$
which we can bound as in the proof of \Cref{pr:general} to obtain 
$$\EE\bigl[ \Gamma' (1)\bigr] \leq  \biggl(\frac{2\EE\bigl[A_n\bigr]{\rm Var}_{{P\otimes Q}} [\rho_\eps]}{n}  \biggr)^{\frac{1}{2}}. $$
Therefore, to conclude the proof, it suffices to solve the inequality 
$$  \EE\bigl[A_n\bigr] - \bigl(\EE\bigl[A_n\bigr]\bigr)^{\frac{1}{2}} A - B \leq 0$$
with $$A={\frac{8\sqrt{2}\,\eps}{t_\eps\, \beta_{n,\eps}  }} \biggl(\frac{{\rm Var}_{{P\otimes Q}} [\rho_\eps]}{n}  \biggr)^{\frac{1}{2}} $$
and 
$$  B=\frac{ 4\bigl(5 \|c\|_\infty + \eps \phi'(1)\bigr)^2}{n}  \biggl(\frac{1}{\eps^u} + \frac{C}{\Bigl(\min_{x\in \Omega}  P\Bigl(\mathbb{B}_{\frac{\delta_\psi \eps}{8L}} (x)\Bigr) \Bigr)^2 } + \frac{C}{\Bigl(\min_{y\in \Omega'}  {Q}\Bigl(\mathbb{B}_{\frac{\delta_\psi \eps}{8L}} (y)\Bigr) \Bigr)^2 } \biggr).$$
Indeed, this inequality implies that $(\EE[A_n])^{\frac{1}{2}}$ is bounded from above by
\begin{align*}
      & \frac{A+ \sqrt{A^2+4 B}}{2} \\
    & \leq  A+\sqrt{B}\\
     &= {\frac{8\sqrt{2}\,\eps}{t_\eps\, \beta_{n,\eps}  }} \biggl(\frac{{\rm Var}_{{P\otimes Q}} \bigl[\rho_\eps  \bigr]}{n}  \biggr)^{\frac{1}{2}}\\
     &\qquad + \biggl( \frac{ 4\bigl(5 \|c\|_\infty + \eps \phi'(1)\bigr)^2}{n}  \biggl(\frac{1}{\eps^u} + \frac{C}{\Bigl(\min_{x\in \Omega}  P\Bigl(\mathbb{B}_{\frac{\delta_\psi \eps}{8L}} (x)\Bigr) \Bigr)^2 } + \frac{C}{\Bigl(\min_{y\in \Omega'}  {Q}\Bigl(\mathbb{B}_{\frac{\delta_\psi \eps}{8L}} (y)\Bigr) \Bigr)^2 } \biggr)\biggr)^{\frac{1}{2}}.
\end{align*}
Hence,
\begin{align*}
 \EE\bigl[A_n\bigr] & \leq  \frac{{2^8} \eps^2}{t_\eps^2\, (\beta_{n,\eps})^2  } \frac{{\rm Var}_{{P\otimes Q}} [\rho_\eps]}{n} \\
 &\quad+  \frac{ 8\bigl(5 \|c\|_\infty + \eps \phi'(1)\bigr)^2}{n}  \biggl(\frac{1}{\eps^u} + \frac{C}{\Bigl(\min_{x\in \Omega}  P\Bigl(\mathbb{B}_{\frac{\delta_\psi \eps}{8L}} (x)\Bigr) \Bigr)^2 } + \frac{C}{\Bigl(\min_{y\in \Omega'}  {Q}\Bigl(\mathbb{B}_{\frac{\delta_\psi \eps}{8L}} (y)\Bigr) \Bigr)^2 }  \biggr)\\
 &= \frac{2^{8} \max\left(\frac{2^{6}\bigl(5\|c\|_{\infty} + \eps \phi'(1)\bigr)^2}{\delta_\psi^2}, 1\right) }{ (\beta_{n,\eps})^2  } \frac{{\rm Var}_{{P\otimes Q}} [\rho_\eps]}{n} \\
 &\quad +  \frac{ 8\bigl(5 \|c\|_\infty + \eps \phi'(1)\bigr)^2}{n}  \biggl(\frac{1}{\eps^u} + \frac{C}{\Bigl(\min_{x\in \Omega}  P\Bigl(\mathbb{B}_{\frac{\delta_\psi \eps}{8L}} (x)\Bigr) \Bigr)^2 } + \frac{C}{\Bigl(\min_{y\in \Omega'}  {Q}\Bigl(\mathbb{B}_{\frac{\delta_\psi \eps}{8L}} (y)\Bigr) \Bigr)^2 }  \biggr),
\end{align*} 
concluding the proof of \cref{theorem:potentials}. 
\end{proof}

\subsection{Proof of \Cref{lemma:Bound-Gamma}}
\begin{proof}[Proof of \Cref{lemma:Bound-Gamma}]
 For $p=2$, $|\psi''_2|\leq 1$, which implies
 \[
 \frac{1}{n}\sum_i \pi_{i,j}^{(1)}\leq 1
 \qquad\text{and}\qquad
 \frac{1}{n}\sum_j \pi_{i,j}^{(1)}\leq 1.
 \]
 For $p=1$, we have
 \[
\frac1n\sum_i \pi_{i,j}^{(1)}
=
\frac{\exp\bigl( \frac{1}{8}\bigr)}{n}\sum_i
\exp\biggl(\frac{\widehat f(X_i)+\widehat g(Y_j)-c(X_i,Y_j)}{\eps}-1\biggr)
=
\exp\left( \frac{1}{8}\right),
\]
where the last equality follows from the empirical first-order condition. Similarly,
\[
\frac1n\sum_j \pi_{i,j}^{(1)}=\exp\left(\frac18\right).
\]
Thus the claim holds for $p=1$ as well.

It remains to consider $p\in (1,2)$. Then the conjugate exponent $q$ belongs to
$(2,\infty)$. We have $t_{\psi_p}=1$ and $\delta_{\psi_p}=1/2$. Moreover,
\[
\psi_p''(t)=(q-1)(t)_+^{q-2}.
\]
Hence
\[
\pi_{i,j}^{(1)}
= (q - 1)\biggl(
\frac{\widehat{f}(X_i)+\widehat{g}(Y_j)-c(X_i,Y_j)}{\eps}
+ \frac18
\biggr)^{q-2}_+ .
\]
We use the elementary bound
\begin{equation}
    \label{eq:elementary}
\left(a+\frac18\right)_+^{q-2}
\leq C_q\left(a_+^{q-2}+1\right),
\qquad a\in\mathbb R,
\end{equation}
for some finite constant $C_q$ depending only on $q$. Indeed, since
\[
\left(a+\frac18\right)_+\leq a_+ + \frac18,
\]
the claim \eqref{eq:elementary} follows from the standard inequalities
\[
(x+y)^\alpha\leq x^\alpha+y^\alpha,\qquad 0<\alpha\leq 1,
\]
and
\[
(x+y)^\alpha\leq 2^{\alpha-1}(x^\alpha+y^\alpha),\qquad \alpha\geq 1,
\]
applied with $\alpha=q-2$, together with $(1/8)^\alpha\leq 1$.
Therefore,
\[
\pi_{i,j}^{(1)}
\leq
(q-1)C_q
\left[
\biggl(
\frac{\widehat{f}(X_i)+\widehat{g}(Y_j)-c(X_i,Y_j)}{\eps}
\biggr)_+^{q-2}
+1
\right].
\]
Consequently,
\begin{align*}
     \frac{1}{n} \sum_{i=1}^n\pi_{i,j}^{(1)}
     &\leq
     (q-1)C_q
     \left[
     \frac{1}{n}\sum_{i=1}^n
     \biggl(
     \frac{\widehat{f}(X_i)+\widehat{g}(Y_j)-c(X_i,Y_j)}{\eps}
     \biggr)_+^{q-2}
     +1
     \right].
\end{align*}
Fix $s= \frac{q-1}{q-2}$. By Jensen's inequality,
\begin{align*}
  \frac{1}{n}\sum_{i=1}^n
  \biggl(
  \frac{\widehat{f}(X_i)+\widehat{g}(Y_j)-c(X_i,Y_j)}{\eps}
  \biggr)_+^{q-2}
  &\leq
  \biggl(
  \frac{1}{n}\sum_{i=1}^n
  \biggl(
  \frac{\widehat{f}(X_i)+\widehat{g}(Y_j)-c(X_i,Y_j)}{\eps}
  \biggr)_+^{(q-2)s}
  \biggr)^{1/s} \\
  &=
  \biggl(
  \frac{1}{n}\sum_{i=1}^n
  \biggl(
  \frac{\widehat{f}(X_i)+\widehat{g}(Y_j)-c(X_i,Y_j)}{\eps}
  \biggr)_+^{q-1}
  \biggr)^{\frac{q-2}{q-1}} \\
  &= 1,
\end{align*}
where the last equality follows from the empirical first-order condition in the $X$-variable.
Hence
\[
\frac{1}{n} \sum_{i=1}^n\pi_{i,j}^{(1)}
\leq 2(q-1)C_q .
\]
The same argument, using the empirical first-order condition in the $Y$-variable, yields
\[
\frac{1}{n} \sum_{j=1}^n\pi_{i,j}^{(1)}
\leq 2(q-1)C_q .
\]
Thus both averages are bounded by a finite constant depending only on $p$, and the result follows.
\end{proof}

\end{document}